\documentclass[11pt,reqno]{amsart}
\usepackage{graphicx,amsmath,amsfonts,latexsym,amssymb,amsthm,color}
\usepackage{latexsym}
\usepackage{verbatim}
\usepackage{enumerate}
\usepackage{tikz-cd}

\usepackage[a4paper]{geometry}

\newtheorem{thm}{Theorem}
\newtheorem{prop}{Proposition}
\newtheorem{lem}{Lemma}
\newtheorem{cor}{Corollary}
\newtheorem{defn}{Definition}

\newtheorem{example}{Example}

\newtheorem{rem}{Remark}

\renewcommand{\Re}{\operatorname{Re}}

\numberwithin{equation}{section}
 
\title{Observability for Schr\"odinger equations with quadratic Hamiltonians}

\author[A. Waters]{Alden Waters}
\address{ University of Groningen, Bernoulli Institute,
Nijenborgh 9,
9747 AG Groningen,
The Netherlands} \email{a.m.s.waters@rug.nl}
\begin{document}

\begin{abstract}
We consider time dependent harmonic oscillators and construct a parametrix to the corresponding Schr\"odinger equation using Gaussian wavepackets. This parametrix of Gaussian wavepackets is precise and tractable. Using this parametrix we prove $L^2$ and $L^2\textendash L^{\infty}$ observability estimates on unbounded domains $\omega$ for a restricted class of initial data. This data includes a class of compactly supported piecewise $C^1$ functions which have been extended from characteristic functions. Initial data of this form which has the bulk of its mass away from $\omega^c=\Omega$, a connected bounded domain, is observable, but data centered over $\Omega$ must be very nearly a single Gaussian to be observable. We also give counterexamples to established principles for the simple harmonic oscillator in the case of certain time dependent harmonic oscillators. 
\end{abstract}

\maketitle


\textbf{MSC classes:} 35R01, 35R30, 35L20, 58J45, 35A22

\section{Introduction}
We start by recalling some results for bounded domains. Let $D$, $D_0$ be bounded smooth subdomains of $\mathbb{R}^d$ with $D_0\subset D$. We consider the controlled linear Schr\"odinger equation 
\begin{align}\label{Sbounded}
&i\partial_tu(t,x)=-\Delta u(t,x)+V(t,x)u(t,x)+f(t,x)\mathbf{1}_{D_0}(x) \quad \mathrm{in}\quad (0,T)\times D \\& \nonumber
u(t,x)=0 \quad\mathrm{on}\quad (0,T)\times\partial D\\& \nonumber
u_0(x)=u(0,x) \quad \mathrm{in}\quad D \nonumber
\end{align}
here $u(t,x)$ is the \textit{state}, $f(t,x)$ is the \textit{control}, and $V(t,x)$ is the \textit{potential}. All can be complex functions. It is known that the system is well-posed in $L^2(D)$ with controls in $L^2((0,T)\times D_0)$ for $L^{\infty}((0,T)\times D)$ potentials. 

Starting with \cite{lebeau} it was shown that provided the Hamiltonian flow and the observation set satisfies the Geometric Control Condition, then the solution to the Schr\"odinger equation \eqref{Sbounded} with an $L^{\infty}$ potential is controllable to another state $u_T=u(T,x)$. Later using Carleman estimates, \cite{LT, bp} and \cite{LTZ, LTZ2} reduced the assumptions on the domain for $L^{\infty}$ time independent and time dependent potentials, respectively,  c.f. \cite{zuazua} for a survey of such results.  If a solution is observable, then often by duality the solution to the adjoint equation is controllable \cite{lions}. Interior observability typically amounts to a solution $u$ satisfying an estimate of the form 
\begin{align*}
\|u_0\|_{L^2(D)}\leq C_T\|u\|_{L^2((0,T)\times D_0)}
\end{align*}
for some nonzero $C_T$, $0<T<\infty$ ($f=0$). Most of the references mentioned here establish observability estimates.

While much is known about the observability problem for Schr\"odinger equations on a bounded domain with  $L^{\infty}$ potentials, much less is known for general operators in the free space, that is when \eqref{Sbounded} has $D$ replaced with $\mathbb{R}^d$ and $D_0$, the support of the control, is replaced with $\omega=\mathbb{R}^d\setminus\overline{\Omega}$ where $\Omega$ is a bounded domain. Correspondingly, one wants to  establish observability for 
\begin{align}\label{hamq}
&i\partial_tu=\left(-\kappa_1(t)\Delta+\kappa_2(t)|x|^2\right)u \quad 0<t\leq T,\quad x\in\mathbb{R}^d\\& \nonumber
u(0,x)=u_0(x).
\end{align}
where $\kappa_1\in C^1([0,T]), \kappa_2\in C([0,T])$, and $u(0,x)$ is real valued and in a weighted $L^2(\mathbb{R}^d)$ space. This amounts to establishing a bound for the solution
of \eqref{hamq} as 
\begin{align*}
\|u_0\|_{L^2(\mathbb{R}^d)}\leq C_T\|u\|_{L^2((0,T)\times \omega)}
\end{align*}
which is challenging as $\omega$ is no longer bounded. While it seems that this would be an easier problem because the background space is $\mathbb{R}^d$, the potential itself is also not compactly supported. The Schr\"odinger operator in the free space with a non-compactly supported potential behaves much differently.  The only results for control and observability for this problem in arbitrary dimensions are in \cite{jkarel} for the case of the simple harmonic oscillator $\kappa_1(t)=\kappa_2(t)=1/2$.  Our goal is then to establish observability for a more general class of time-dependent harmonic oscillators. We end up establishing an approximate observability theorem for a limited class initial data $u_0$, and an observability estimate for some examples of 'Gaussian-like' initial data. Approximate observability here means that $u(t,x)$ satisfies
\begin{align}\label{etao}
\|u_0\|_{L^2(\mathbb{R}^d)}-\eta\leq C_T(\|u\|_{L^2((0,T)\times \omega)}+T\eta)
\end{align}
where $\eta>0$ can be made very small. This error term has to do with the fact that translated real Gaussians are dense in $L^1(\mathbb{R}^d)$ and hence can approximate $L^2\cap L^{\infty}(\mathbb{R}^d)$ functions, but are not a basis for either space. Whenever $\eta$ depending on the $L^{\infty}(\mathbb{R}^d)$ norm of $u_0$ is small enough that the $\eta$ 'errors' in \eqref{etao} are absorbed into the $L^2(\mathbb{R}^d)$ norm of $u_0$ we call this an $L^2\textendash L^{\infty}(\mathbb{R}^d)$ observability estimate. 

 The underlying idea here is the use of sophisticated Gaussian wave packets. We prove an explicit approximation theorem which decomposes $f\in L^2(\mathbb{R}^d,e^{x^2})$ into a finite sum of Gaussians up to some quantifiable error, which is based on \cite{calc}. Thus there are two main goals here, one is to construct a nearly explicit parametrix using this decomposition and the other is to prove some localization properties of solutions to the Schr\"odinger equation in the form of observability estimates. While a least squares approximation may require fewer terms with a higher degree of accuracy, we give an explicit characterization of the decomposition needed to create a completely explicit parametrix with a high degree of accuracy. This parametrix in Theorem \ref{parametrix} was largely inspired by the Fourier integral operators developed in \cite{karel}. The derivation of the theorem constructing the parametrix is based on properties of Hermite functions and polynomials.  
  
 
To see why this is important, we remark that \emph{any} quantum mechanical system with a potential energy $V(x)$ has local equilibrium points which can be analysed by the model for a quadratic Hamiltonian, \eqref{hamq} with $\kappa_1(t)=\kappa_2(t)=1/2$. In other words, Taylor expanding $V(x)$ around the point $x_0$ gives:
\begin{align*}
V(x)\approx V(x_0)+\nabla V(x_0)(x-x_0)+\frac{1}{2}\nabla^2V(x_0)(x-x_0)^2.
\end{align*}
If $x_0$ is a critical point, then the second term vanishes. Translating $x_0$ to zero we obtain:
\begin{align*}
V(x)\approx V(0)+\frac{1}{2}\nabla^2V(0)x^2,
\end{align*}
and we see that the model is reduced to the one of the harmonic oscillator. The Hamiltonian associated to this operator is $\frac{1}{2}(|p|^2+|x|^2)$. The Hamiltonian ray path for this operator is computed by solving the system of ODE's
\begin{align*}
\frac{dx(t)}{dt}=p(t) \quad \frac{dp(t)}{dt}=-x(t).
\end{align*}
The key point is that this operator has a cusp at zero, e.g. one point for which the Hamiltonian ray path is such that $\frac{dx(t)}{dt}=0$ which can affect the validity of a Gaussian wavepacket construction. Making sense of this phenomenon and the regions where a solution to a much more general time-dependent classical problem \eqref{hamq} are concentrated is one of the goals of this paper. The examples presented here all represent toy models of time dependent metrics and potentials $V(t,x)$ in the free space background case around spatial equilibrium. Establishing parametrices and observability for such potentials is key to understanding the behavior of free space quantum mechanical waves. 

The use of Gaussian wavepackets has been around since the work of \cite{hagedorn} and \cite{cordobafefferman}. For the semi-classical Schr\"odinger equation in the free space
\begin{align} \label{semiclass}
&ih\partial_tu=(-h^2\Delta+V)u \quad \mathrm{in} \quad (0,T)\times \mathbb{R}^d\\& \nonumber
u(0,x)=u_0(x)
\end{align}
where $u_0\in C^{\infty}(\mathbb{R}^d)$, the leading order term of $u$ up to a high degree of accuracy in $h$ has been shown to take the form of a Gaussian Fourier Integral Operator (FIO). This description holds for very general time-dependent Hamiltonians see \cite{robert1, robert2, robert3, Liu} and also \cite{laptevsigal}, and finite times $T$. However for the classical Schr\"odinger equation the situation is a bit different as we do not have the added advantage of calculating errors in terms of the scale of the semi-classical parameter $h$. For control theory results in the semi-classical case, we direct the reader to \cite{Macia1, Macia2, Macia3}. 

We still have the property in the classical case presented here that the FIOs which are constructed for the non-autonomous problem are Gaussian distributions when the Hamiltonian is quadratic c.f. \cite{karel} for a full treatment. Hence the reduction to the model in \eqref{hamq}.  A Gaussian FIO applied to a Gaussian function is again, another Gaussian, which motivates our choice of approximation. The FIO construction now allows us to create a parametrix solution to the Schr\"odinger equation consisting of a finite collection of tractable propagated wavepackets whose properties, while technical, are not impossible to describe. Some same principles from the compact case related to the Geometric Control Condition still carry over. Moreover, because the representation used here is very close to the classical Hermite functions, the analysis of inner products on $L^2(\omega)$ spaces has a direct correlation to spectral theory analysis. Using separation of variables it is usually desirable to prove $\langle\psi_n,\psi_m\rangle_{L^2(\omega)}$ is small in order that the $\|\psi_n\|^2_{L^2(\omega)}$ terms dominate for some $\psi_n$ a spectral basis for the underlying space ($\mathbb{R}^d$, in this case and $\omega=\mathbb{R}^d\setminus \overline{\Omega}$ with $\Omega$ a bounded domain). We see the analogue of this idea directly for the dynamic wavepackets in Lemma \ref{lowerinner} and the proof of Theorem \ref{main3} in the text. Unfortunately this means we have to restrict the class of initial data for the proofs. 

What the theorems here say is that initial data which can be approximated by Gaussians (including some piecewise $C^1$ compactly supported functions which have been extended from characteristic functions) are $L^2\textendash L^{\infty}(\mathbb{R}^d)$ observable far away from the 'hole', $\Omega$ and only initial data which is very nearly a Gaussian is observable when the support of the bulk of the 'mass' ($L^2(\mathbb{R}^d)$ norm) of the initial data sits over the 'hole'.  This makes sense intuitively as initial data which is Gaussian propagates as another Gaussian. Far away from the 'hole' we see most of the mass of the Gaussian if it started off there, and near to the complement of $\Omega$, the 'hole', we recover almost nothing if the mass started over $\Omega$. We give some explicit examples of such initial data and $\kappa_1(t)$ and $\kappa_2(t)$ in Section \ref{apps}. The various values of $\kappa_1(t)$ and $\kappa_2(t)$ can make the behaviour of the position and the spread of the parametrix quite different even though the overall shape is still Gaussian in the spacial variable. It is curious that in the case of the $1d$ harmonic oscillator $\kappa_1=\kappa_2=1/2$ that observability on these types of domains $\omega$ is true for all $f\in L^2(\mathbb{R})$ if $T\geq \pi$ because the solutions are periodic in time, c.f, Theorem 2.2 \cite{karelhypo}. Therefore in this case our construction still has a gap to be filled because the parametrix here is only valid until $T=\frac{\pi}{2}$ for $\kappa_1=\kappa_2=1/2$. While this parametrix could be extended to $T\geq \pi$ the problem is still that the class of data which is approximately observable depends on location of its support. This gap in our proof techniques cannot easily be closed because it would involve analyzing lower bounds on $\mathrm{erfc}(z)$ for complex $z$. However there are no known results on observability when $\kappa_1(t)$ and $\kappa_2(t)$ both depend on time. 

 For the free Schr\"odinger equation and classical harmonic oscillator, the basis of Hermite functions and principles of spectral decomposition have already resulted in our Theorem \ref{main} with $\eta=0$ in \cite{miller, jkarel, wangwang, wangwang2, kp}. However one does not expect these techniques to extend themselves to parametrices or observability in the case of time dependent operators as these are purely time separable techniques. Indeed, to this end some counterexamples are shown to established principles for the time independent simple harmonic oscillator in Example \ref{main2}.  Therefore the theorems here to try new techniques which may be applicable to a more general class of operators which are not accessible with spectral theory directly. 
 
\section{Background and Main Theorems}
\subsection{Time dependent quadratic operators} We consider a class of time-dependent quadratic operators:
\begin{align*}
L=-\kappa_1(t)\Delta+\kappa_2(t)|x|^2
\end{align*}
where $0\leq \kappa_2(t)$ depends continuously on time $t$ for $0\leq t\leq T$, and for simplicity we also assume $\kappa_1(t)$ is bounded below by a positive constant and $\kappa_1\in C^1([0,T])$. 
We are interested in observability for the non-autonomous Cauchy problem
\begin{align}\label{nonauto}
&\partial_tu(t,x)+iLu(t,x)=0 \quad 0<t\leq T,\quad x\in\mathbb{R}^d\\& \nonumber
u_0=u(0,x).
\end{align} 
Because of the non $L^2(\mathbb{R}^d)$ term $|x|^2$ in the operator, we need a different definition of well-posedness. To this end we set 
\begin{align*}
B=\{u\in L^2(\mathbb{R}^d): x^{\alpha}D_x^{\beta}u\in L^2(\mathbb{R}^d),\,\, \alpha, \beta\in \mathbb{N}^d_0, \,\, |\alpha+\beta|\leq 2\}. 
\end{align*} 
The space $B$ is a Hilbert space equipped with the norm
\begin{align*}
\|u\|_{B}^2=\sum\limits_{\substack{\alpha,\beta\in \mathbb{N}_0^d\\|\alpha+\beta|=2}}\|x^{\alpha}D_x^{\beta}u\|_{L^2(\mathbb{R}^d)}^2.
\end{align*}
Because there is no easy condition that guarantees the existence of classical solutions for non-autonomous Cauchy problems, it was shown in \cite{karel}, that the equation \eqref{nonauto} has a B-valued solution and that this solution exists and is unique. We recall the definition of a B-valued solution. 
\begin{defn}[B-valued solutions \cite{karel}]
A continous function $u\in C^0([0,T],B)$ is a B-valued solution of the non-autonomous Cauchy problem \eqref{nonauto} if $u\in C^1([0,T], L^2(\mathbb{R}^d))$, then equation \eqref{nonauto} is satisfied in $L^2(\mathbb{R}^d)$. 
\end{defn}
 A portion of Theorem 1.2, by Pravda-Starov in \cite{karel} gives that for every $u_0\in B$, the non-autonomous Cauchy problem \eqref{nonauto} and its adjoint have a unique B-valued solution, and this solution is unitary for all $t\in [0,T]$. Theorem 1.2 of \cite{karel} is actually more general and includes complex-valued quadratic operators $iL$ with a non-positive real part for their Weyl symbol. For time independent quadratic Hamiltonians, the theorem is due to H\"ormander \cite{hormandermehler}. We could also analyze complex-valued quadratic operators with the same methodology presented here, but the computations for completely general operators become very difficult quickly, unless a specific example is specified.  

\subsection{Statement of the Main Theorems}
We recall that the classical Wiener's Tauberian theorem in \cite{wiener} says that the span of translates of functions are dense in $L^1(\mathbb{R})$ if they have Fourier transform which is non vanishing everywhere. This theorem implies that for $f\in L^1(\mathbb{R})$ and $\eta\in (0,1)$ that there exists a finite number $N$ and real numbers $\alpha_n$ and $\beta_n$ such that 
\begin{align}
\int\limits_{\mathbb{R}}\left|f(x)-\sum\limits_{|n|\leq N}\beta_n e^{-(x-\alpha_n)^2}\right|\,dx<\eta.
\end{align}
It does not tell what the coefficients $\beta_n$ are nor what the $\alpha_n$ might be. This statement can be generalized to higher dimensions. Furthermore since real Gaussians cannot be made into an orthonormal basis for $L^2(\mathbb{R}^d)$, there is also no least squares method to generate the $\beta_n$ and $\alpha_n$. In order to build an explicit parametrix for $L^2(\mathbb{R}^d)$ functions, we start by approximating them in a different way in Theorem \ref{approx2} below. Theorem \ref{approx2} can be thought of as more explicit than the implication of Wiener's Tauberian theorem applied to estimate $L^2(\mathbb{R}^d)\cap L^{\infty}(\mathbb{R}^d)$ functions since it specifies the coefficients and errors needed for approximation of $L^2(\mathbb{R}^d,e^{x^2})$ functions. 

Let $h_k(x)$ be the $k^{th}$ Hermite function. We prove the following Theorem expanding on the work of \cite{calc}. 
\begin{thm}\label{approx2}
Let $f$ be in $L^2(\mathbb{R}^d,e^{x^2})$, and $N$ a fixed natural number with $N>2$ and $\epsilon_0\in (0,1)$. Let $n=(n_1, . .. ,n_d)$ and $k=(k_1, . . .,k_d)$ be such that $n_i, k_i\in \{0, . . .,N\}$ for all $i$ and define
\begin{align}\label{cn}
&c_n=\\& 
\frac{1}{n_1!\sqrt{\pi}}\sum\limits_{k_1=n_1}^{N_1}\frac{(-1)^{n_1-k_1}}{(k_1-n_1)!(2\epsilon_0)^{k_1}} \,.\, .\, . \frac{1}{n_d!\sqrt{\pi}}\sum\limits_{k_d=n_d}^{N_d}\frac{(-1)^{n_d-k_d}}{(k_d-n_d)!(2\epsilon_0)^{k_d}}\int\limits_{\mathbb{R}^d}f(x)e^{x^2}\frac{\partial^{k}}{\partial x^{k}}e^{-x^2}\,dx \nonumber
\end{align}
for $N_i \in \{0, . . .,N\}$ such that $|(N_1, . . .,N_d)|=N$. We then have that 
\begin{align*}
&\left(\int\limits_{\mathbb{R}^d}\left|f(x)-\sum\limits_{|n|\leq N}c_n e^{-|x+n\epsilon_0|^2}\right|^2\,dx\right)^{\frac{1}{2}}\leq (e^NN\epsilon_0)^d\|f\|_{L^2(\mathbb{R}^d)}+E_N. \nonumber
\end{align*} 
The term $E_N$ is given by 
 \begin{align*}
E_N=\sum \limits_{|n|=N+1}^{\infty}|d_n|^2 \quad\quad d_n=\int\limits_{\mathbb{R}^d}f(x)e^{\frac{x^2}{2}}h_n(x)\,dx.
\end{align*}
\end{thm} 
In some cases $E_N=0$ such as $f=\sum\limits_{|k|\leq N}\alpha_k h_k e^{-\frac{x^2}{2}}$ with $\alpha_k\in \mathbb{R}$. Bounds on $E_N$ for $H^3([-M,M]^d)$ functions can be found in Proposition \ref{compact1}. 

We use this wavepacket approximation to generate accurate solutions to the Schr\"odinger equation. We must make some assumptions on the time span of the solutions in order for the solution not to have any singularities. As such we make the following definition.
\begin{defn}\label{flow}
 Let $H=\kappa_1(t)|p|^2+\kappa_2(t)|x|^2$ be the Hamiltonian associated to \eqref{nonauto} where $\kappa_1(t), \kappa_2(t)$ are continuous functions. The solution $(x(t),p(t))$  to  
\begin{align*}
\frac{dx_i(t)}{dt}=2\kappa_1(t)p_i(t) \quad \frac{dp_i(t)}{dt}=-2\kappa_2(t)x_i(t) \quad i=1, . . .,d\quad  0\leq t\leq T
\end{align*} 
with $(x(0),p(0))$ such that $x(0)=(1,. . . ,1)$, $p(0)=(0, . . .,0)$ are the Hamiltonian trajectories. We say it is non-zero if for all times $t$, $0\leq t\leq T$, if $x_i(t)\neq 0$ for all $i=1, . . .,d$ and the quantity
\begin{align*}
a(t)=e^{-\int\limits_0^t\frac{p_i(s)\kappa_1(s)}{x_i(s)}\,ds }
\end{align*}
is bounded. This quantity is the amplitude of a one dimensional wavepacket. 
\end{defn}
\begin{prop}\label{maincorollary}
The potential pairs $\kappa_1(t)=\kappa_2(t)=1/2$ (harmonic oscillator) and $\kappa_1(t)=\frac{e^{2at}}{2}, \kappa_2(t)=\frac{e^{-2at}}{2}$, $a>1/2$ (Cadirola-Kanai oscillator) 
have non-zero Hamiltonian flows for $T=\frac{\pi}{2}, \infty$, respectively. Thus they satisfy the conditions of Theorem \ref{main} with these particular $T_D$. For the Hamiltonian with $\sigma$ a constant \begin{align*}
H(t)=\frac{|p|^2}{2(t+d)^a}+\frac{\sigma^2(t+d)^b|x|^2}{2}, \quad a,b,d>0
\end{align*}
it is possible to find a nonzero $T_D$ depending on d,b,a since the solutions to the Hamiltonian ODEs are given in terms of Bessel functions. 
\end{prop} 
Proposition \ref{maincorollary} is shown in Section \ref{apps}. 
We build an explicit parametrix for $0\leq t\leq T$ with $T$ in the definition above for generic time dependent Schr\"odinger equations for initial data $u_0=f(x)e^{-\frac{|x|^2}{2}}$ with $f\in L^2(\mathbb{R}^d)$ in Theorem \ref{parametrix} stated below whose proof is given in Section \ref{wavepacketconstruction}. 
\begin{thm}\label{parametrix}
Let $u_0=f(x)e^{-|x|^2/2}$, with $f$ in $L^2(\mathbb{R}^d)$. Then for some $\eta\in (0,1)$ there is an $N$ sufficiently large such that
\begin{align}\label{approxH}
||u_0-\sum\limits_{|n|\leq N}d_nh_n(x)e^{\frac{-|x|^2}{2}}||_{L^2(\mathbb{R}^d)}\leq \eta
\end{align}
where $d_n=\langle f,h_n\rangle_{L^2(\mathbb{R}^d)}$. 
We let $\epsilon_0\in (0,1)$ be small and $a_n=n\epsilon_0$ with $n=(n_1, . . .,n_d)$, $k=(k_1, . . .,k_d)$ such that $n_i, k_i\in \{0, . . .,N\}$ for all $i$ and define
\begin{align}\label{cdn}
&c_n=\\& 
\frac{1}{n_1!\sqrt{\pi}}\sum\limits_{k_1=n_1}^{N_1}\frac{(-1)^{n_1}}{(k_1-n_1)!(2\epsilon_0)^{k_1}}\, .\, .\, . \frac{1}{n_d!\sqrt{\pi}}\sum\limits_{k_d=n_d}^{N_d}\frac{(-1)^{n_d}}{(k_d-n_d)!(2\epsilon_0)^{k_d}}d_k \nonumber
\end{align}
for $N_i \in \{0, . . .,N\}$ such that $|(N_1, . . .,N_d)|=N$.
Let $u$ be the solution to \eqref{nonauto} for all times $t$ with $0\leq t\leq T$ and $T$ in Definition \ref{flow} with this type of initial data, $u_0$, then
\begin{align}\label{twoterm}
\left(\int\limits_0^T\int\limits_{\mathbb{R}^d}\left|u(t,x)-\sum\limits_{|n|\leq N}c_n\phi_n(t,x)\right|^2\,dx\,dt\right)^{\frac{1}{2}}\leq \left(\eta+(e^{N}N\epsilon_0)^d)\right)T
\end{align}
where
\begin{align*}
\phi_n(t,x)=\left(\frac{a^2(t)}{(1-4iy_3(t))}\right)^{\frac{d}{2}}e^{iy_1(t)|x|^2-\frac{|y_2(t)x+a_n|^2}{(1-4iy_3(t))}}.
\end{align*}
The functions $y_1(t),y_2(t),y_3(t)$ are solutions to the Riccati equations (recalled in \eqref{Riccati} Section \ref{wavepacketconstruction}) associated with the Hamiltonian in Definition \ref{flow}. 
\end{thm}
This parametrix is very explicit as all the functions involved in the representation have an exact form.  Indeed, the $y_1, y_2,y_3$ can be computed in terms of rational functions for the Cadirola Kanai oscillator (Example \ref{CK} in Section \ref{apps} here) and for the harmonic oscillator $$\left(-\frac{\tan(t)}{2},\frac{1}{\cos(t)},-\frac{\tan(t)}{2}\right)$$ (c.f. Exercise 11.1, p.129 \cite{grigissjostrand}) and standard Schr\"odinger equation $$(y_1(t)=0,y_2(t)=1,y_3(t)=-t).$$ 

We now make the following definition about the class of $\mathcal{A}$ which are candidates for being approximately observable. 
\begin{defn}\label{as1} We let the class of admissible functions $\mathcal{A}$ be defined as follows. Let $f\in L^2(\mathbb{R}^d)$ be such that there exists a natural number $N$ with $N>2$, $\eta\in (0,1)$ $\tilde{a}_n\in \mathbb{R}^d$, and $\tilde{c}_n\in\mathbb{R}_+$ so we can approximate $f$ as 
\begin{align}\label{etadef}
&\left(\int\limits_{\mathbb{R}^d}\left|f(x)-\sum\limits_{|n|\leq N}\tilde{c}_ne^{-|x+\tilde{a}_n|^2}\right|^2\,dx\right)^{\frac{1}{2}}\leq  \eta
\end{align} 
where $n=(n_1, . .. ,n_d)$ $n_i\in \{0, . . .,N\}$, $|n|\leq N$. We also assume that $f$ has bounded $B$ norm. 
\end{defn} 
The $\tilde{a}_n$ and $\tilde{c}_n$ in the above definition may or may not coincide with the $a_n$ and $c_n$ given in Theorem \ref{approx2} if we further assume $f\in L^2(\mathbb{R}^d,e^{x^2}\,dx)$. 
We give an elementary example (Lemma \ref{stepg}) in Section \ref{approxg} that shows this class $\mathcal{A}$ includes some compactly supported piecewise $C^1$ functions which have been extended from characteristic functions. We prove the following two approximate observability theorems with $u_0\in\mathcal{A}$. We can think of each theorem having conditions on the spread of the initial data, $$\max\limits_{\substack{n,m\\|m|,|n|\leq N}} |\tilde{a}_n-\tilde{a}_m|,$$ and placement of the support, but not on the coefficients $\tilde{c}_n$ in the definition of $\mathcal{A}$. We say a bounded domain $\Omega\subset \mathbb{R}^d$ is centered at the origin, if the ball containing $\Omega$ has center at the origin. 
\begin{thm}\label{main}
Let $\Omega$ be a bounded domain centered at the origin and assume $\omega=\mathbb{R}^d\setminus\overline{\Omega}$. Assume the Hamiltonian associated to \eqref{nonauto} is non-zero on the interval $[0,T_D]$. Then, there exists a corresponding nonzero constant $C_{T}$ for all $T\in (0,T_D)$ with $T<\infty$ depending on $\Omega$, the $L^{\infty}([0,T])$ norm of $\kappa_2(t)$, and the $C^1([0,T])$ norm of $\kappa_1(t)$ such that solutions to the non-autonomous evolution equation \eqref{nonauto} with initial data $u_0\in \mathcal{A}$ satisfy the following inequality 
\begin{align}\label{obs}
\|u_0\|_{L^2(\mathbb{R}^d)}-\eta\leq C_{T}\left(\|u\|_{L^2((0,T)\times\omega)}+T\eta\right)
\end{align} 
provided $\tilde{a}_n=\varepsilon n$ with $\varepsilon$ sufficiently small. In particular, $\varepsilon$, for fixed $N$ satisfies inequality \eqref{req} essentially giving $\varepsilon\leq \frac{e^{-C_1N}}{C_2N}$ for $C_1$, $C_2$ constants depending on time and $\Omega$ when $2N<\mathrm{diam}(\Omega)$.
\end{thm}
Note that we explicitly describe $C_T$ in the text for all values of $\mathrm{diam}(\Omega)$ and give examples which show this condition \eqref{req} gives a nonempty set of $\varepsilon$ for each $N$ and $T$ and $\Omega$. We call this an approximate observability inequality since spectral representations such in \cite{karelhypo} and \cite{karelparabolic} in the literature use exact representations of $u_0\in L^2(\mathbb{R}^d)$ to establish their observability inequalities. Here our $\eta$, even though it is small, can depend on another norm of the initial data such as the $L^{\infty}$ norm, see Lemma \ref{stepg} in Section \ref{approxg}. 
\begin{example}
Take $u_0=\sum\limits_{|k|\leq N} b_k e^{-|x+k\varepsilon|^2}$ where $b_k\geq 0$ are (any) real positive numbers. Consider $\Omega$ to be the ball of radius $R>1$. We have that $\eta=0$, and since $u_0$ is in $\mathcal{A}$
\begin{align*}
\|u_0\|_{L^2(\mathbb{R}^d)}\leq C_{T}\|u\|_{L^2((0,T)\times\omega)}
\end{align*}
holds provided \eqref{req} holds. This inequality \eqref{req} gives $\varepsilon\leq \frac{e^{-C_1NR}}{C_2N}$ where $C_1$ and $C_2$ are constants depending on time for $N<R$. These constants are computed in Section \ref{apps} for various pairs of $\kappa_1(t)$ and $\kappa_2(t)$. 
\end{example}
While it may seem that this class of functions is very small, and indeed it is, there is no known observability results for these types of oscillators in which \emph{both} $\kappa_1(t)$ and $\kappa_2(t)$ depend on time. Furthermore this says that if the bulk of the $L^2(\mathbb{R}^d)$ initial data class sits over the origin, over the center of $\Omega$-- the 'hole" then they must be very nearly perturbations of a Gaussian to be observable. Intuitively, the second theorem says that if we stay far away from $\Omega$ then the initial data does not reach the 'hole'. There is only one similar result in the literature in the $d=1$ case for $\kappa_1=\kappa_2=1/2$ which is Theorem 1.5 of \cite{gensheng}. Their theorem states that the support of solution must intersect the observation region in a way which is bounded below by a Gaussian to be $L^2(\mathbb{R}^d)$ observable. 
\begin{thm}\label{main3}
Let $\Omega$ be a bounded domain centered at the origin and assume $\omega=\mathbb{R}^d\setminus\overline{\Omega}$. Assume the Hamiltonian associated to \eqref{nonauto} is non-zero on the interval $[0,T_D]$. Then, there exists a corresponding nonzero constant $C_{T}$ for all $T\in (0,T_D)$, with $T<\infty$, depending on $\Omega$, the $L^{\infty}([0,T])$ norm of $\kappa_2(t)$, and the $C^1([0,T])$ norm of $\kappa_1(t)$ such that solutions to the non-autonomous evolution equation \eqref{nonauto} with initial data $u_0\in \mathcal{A}$ satisfy the following inequality 
\begin{align}\label{obs}
\|u_0\|_{L^2(\mathbb{R}^d)}-\eta\leq C_{T}\left(\|u\|_{L^2((0,T)\times\omega)}+T\eta\right)
\end{align} 
provided that for $|\alpha_N|=\max\limits_{\substack{n,m\\|m|,|n|\leq N}} |\tilde{a}_n-\tilde{a}_m|$ we have 
 \begin{align}\label{R1}
\max\limits_{n}|\tilde{a}_n|>\sqrt{\frac{|\alpha_N|^2+2-2\log(\sqrt{|y_2(t)|(1+16(y_3(t))^2)^{-1}})}{4(y_2(t))^2(1+16(y_3(t))^2)^{-1}}}+\frac{|\alpha_N|}{y_2(t)}+\frac{\mathrm{diam}(\Omega)}{2}
 \end{align}
 The functions $y_2(t),y_3(t)$ are solutions to the Riccati equations (recalled in \eqref{Riccati} Section \ref{wavepacketconstruction}) associated with the Hamiltonian in Definition \ref{flow}
\end{thm}
Functions which are $L^{2}\textendash L^{\infty}(\mathbb{R}^d)$ observable are examined in Section \ref{conclusions} in Corollary \ref{linftyobs}. 
This theorem shows that Gaussian like data sufficiently far away from the 'hole" $\Omega$ is nearly observable. We also compute examples in Section \ref{apps} which show that the right hand side of \eqref{R1} is bounded and therefore this condition holds for a nonempty set of $u_0\in\mathcal{A}$. The idea behind the above theorem also lends itself to counterexamples when $\Omega$ is no longer bounded.
We recall the following definition which is a modification from the one in \cite{jkarel}.
\begin{defn}\label{geof} [Geometric Filling]
Let $\omega$ be a measurable subset of $\mathbb{R}^d$. We say that $\omega$ satisfies the geometric filling condition if
\begin{align*}
\liminf_{R_1\rightarrow +\infty} . . . \liminf_{R_d\rightarrow +\infty} \frac{|\omega\cap [-R_1,R_1]\times . . \times [-R_d,R_d]|}{|[-R_1,R_1]\times . . .\times [-R_d,R_d]|}>0
\end{align*}
holds. 
\end{defn} 
Given this definition, we have the following counterexample(s). 
\begin{example}\label{main2}
If the Geometric Filling condition in Definition \ref{geof} fails, then $\omega$ is not necessarily observable for all $u_0\in B$ in any dimension $d\geq 2$. However, the condition is not sufficient to guarantee observability of \eqref{nonauto} for all time dependent operators. In particular it is not enough to guarantee the observability inequality holds when the Hamiltonian flow does not change sign. 
\end{example}
The demonstration of the above statement is based on two counterexamples. We now remark that Theorem \ref{main} is present in the literature but only for the free space Schr\"odinger equation and the simple harmonic oscillator. 
\begin{rem} 
In \cite{jkarel} the authors prove exact control for the time-independent fractional Schr\"odinger equation
\begin{align*}
&(\partial_t-i(-\Delta+|x|^2)^s)u=0\\&
u(0,x)=u_0(x)
\end{align*} 
for $s\geq 1$ by giving very precise spectral estimates for $s\geq 1/2$ and $d\geq 1$, but only general principles for the time of observability. We are not able to handle the case when $s\neq 1$ here, since our methods are based on FIO solutions which do not exist for fractional harmonic oscillators, but we are able to say something about general time dependent quadratic Schr\"odinger equations. In contrast, the authors of \cite{gensheng} analyze control for 
\begin{align*}
&(\partial_t-i(-\Delta+|x|^{2m})u=0 \quad  m\in \mathbb{N} \\&
u(0,x)=u_0(x).
\end{align*}
We would expect that in the case $m>1$ FIO solutions also exist, so an extension to time dependent operators would also be possible. 
\end{rem}

Many of the lemmata and theorems presented in this text for building the solution are also applicable to the time-dependent heat equation under the Wick rotation $t\mapsto it$. Observability for the corresponding time-independent heat equation was considered using a spectral decomposition \cite{karelparabolic} and for the free space heat equation in \cite{wangwang3}. However at a certain critical points in the estimates we use the unitarity of the Schr\"odinger equation to establish observability, and it is unclear what the corresponding replacement analogues are to this property for the heat equation. In specific, we need unitarity of the Schr\"odinger equation for finding lower bounds for the parametrix solutions. However the parametrix estimates still hold under the Wick rotation.

\section{Construction of the FIO for generic $\kappa(t)$}

In this section we construct explicit examples of FIO solutions to \eqref{nonauto} which are only abstractly constructed in \cite{karel}. Our treatment is similar to \cite{grigissjostrand}, c.f. also \cite{hormander3}.  
\begin{lem}\label{FIOlem}
Assume the Hamiltonian flow (Definition \ref{flow})  is non-zero for all $0\leq t\leq T$. The associated FIO solution to \eqref{nonauto}  can be written as 
\begin{align*}
u(t,x)=\frac{1}{(2\pi)^{\frac{d}{2}}}\int\limits_{\mathbb{R}^d}e^{i\phi(t,x,\eta)}(a(t))^d\hat{u}_0(\eta)\,d\eta.
\end{align*}
The equality is understood for non-$L^1$ $u_0$ functions as in the sense of distributions as above. 
The phase is of the form 
\begin{align*}
i\phi(t,x,\eta)=iy_1(t)|x|^2+iy_2(t)x\cdot\eta+iy_3(t)|\eta|^2
\end{align*}
where $y_1(t), y_2(t), y_3(t)$ are functions of $t$ as determined by the following system of ordinary differential equations
\begin{align}\label{system}
&y'_1(t)=-4\kappa_1(t)(y_1(t))^2-\kappa_2(t) \quad y_1(0)=0\\ \nonumber
&y_2'(t)=-4\kappa_1(t)y_1(t)y_2(t)\quad y_2(0)=1\\ \nonumber
&y_3'(t)=-\kappa_1(t)y_2^2(t) \quad y_3(0)=0 \nonumber
\end{align}
and the amplitude satisfies
\begin{align*}
a(t)=e^{-2\int\limits_0^t y_1(s)\kappa_1(s)\,ds}.
\end{align*}
\end{lem}
\begin{proof} 
By Theorem 1.3 in \cite{karel}, as discussed in the previous section, the phase function $\phi(t,x,\eta)$ can be written in the form 
\begin{align*}
y_1(t)|x|^2+y_2(t)x\cdot\eta+y_3(t)|\eta|^2
\end{align*}
where $y_1(t), y_2(t),y_3(t)$ are continuous functions of time for $0\leq t\leq T$. 
The phase function will necessarily solve the eikonal equation 
\begin{align*}
\partial_t\phi+\kappa_1(t)|\nabla_x\phi|^2+\kappa_2(t)|x|^2=0.
\end{align*}  
This leads to the following system of ordinary differential equations 
\begin{align}\label{Riccati}
&-y'_1(t)=4\kappa_1(t)(y_1(t))^2+\kappa_2(t)\\ \nonumber
&-y_2'(t)=4\kappa_1(t)y_1(t)y_2(t)\\ \nonumber
&-y_3'(t)=\kappa_1(t)y_2^2(t). \nonumber
\end{align}
Using the condition $\nabla_x\phi(0,x)=\eta$ and assuming $y_1(t)$ is purely real, we obtain the initial conditions for the system. It remains to solve the the transport equation to find $a(t)$. We then have that
\begin{align*} 
a'(t)=-2y_1(t)\kappa_1(t)a(t) \quad a(0)=1
\end{align*}
and the proof is finished. 
\end{proof}
\begin{lem}\label{y1}
The solution $y_1(t)$ to the Riccati equation in \eqref{system} (first equation) is well-posed when the Hamiltonian flow is non-zero. 
\end{lem}
\begin{proof}
There are examples of $\kappa_1(t),\kappa_2(t)$ for which the Riccati equation is not well-posed (either at all, or after a certain interval) and a corresponding FIO solution cannot be constructed. While this is implied by the abstract constructions in \cite{hormandermehler}, \cite{hormanderq} and \cite{karel}, it is important to see explicitly in our section on observability estimates. Let $S(t)=4\kappa_1(t)\kappa_2(t)$, and $R(t)=\frac{d}{dt}\mathrm{ln}(\kappa_1(t))$. Then a solution $v(t)$ to 
\begin{align*}
v''-Rv'+Sv=0 \quad v(0)=1,\quad v'(0)=0
\end{align*}
is related to the Riccati solution (first equation in \eqref{system}) by $y_1=\frac{v'}{4\kappa_1 v}$. Therefore in order for the Riccati equation to be well-posed, $\kappa_1(t)$ must also be differentiable and nonzero, and $v$ must not be zero. (In addition to $\kappa_2(t)$ being continuous). Notice that in our case $x(t)=v$, which is why we assume the Hamiltonian flow is non-zero.  
\end{proof}
Theorem 1.3 in \cite{karel} then gives that the FIO is unitary when the flow is non-zero. This statement also follows by direct calculation in our case.  Notice that the second assumption in the definition is equivalent to $a(t)^2<\infty$. 

\section{Proof of Theorem \ref{approx2}, Approximation with Gaussians}\label{approxg}
We would like to use the following Hermite inspired expansion for all $f(x)$ in $L^2(\mathbb{R},e^{x^2}\,dx):$ 
\begin{align*}
f(x)=\sum\limits_{n=0}^{\infty} b_n \frac{d^n}{dx^n}e^{-x^2}
\end{align*}
where 
\begin{align*}
b_n=\frac{1}{2^nn!\sqrt{\pi}}\int\limits_{\mathbb{R}}f(x)e^{x^2}\frac{d^n}{dx^n}e^{-x^2}\,dx.
\end{align*}
 Indeed, recall that the Hermite polynomials are defined as 
\begin{align}\label{Hermitepoly}
H_n(x)=(-1)^ne^{x^2}\frac{d^n}{dx^n}e^{-x^2}.
\end{align} 
Then the set of Hermite functions 
\begin{align}\label{Hermite}
\left\{ h_n(x)=\frac{1}{\sqrt{2^n n! \sqrt{\pi}}}H_n(x)e^{-x^2/2}: n\in \mathbb{N} \right\}
\end{align}
is a well known orthonormal basis of $L^2(\mathbb{R})$. For any $g\in L^2(\mathbb{R})$ it is possible to write
\begin{align*}
g(x)=\sum\limits_{n=0}^{\infty} d_n h_n(x)
\end{align*}
where 
\begin{align*}
d_n=\int\limits_{-\infty}^{\infty}g(x)\overline{h_n(x)}\,dx.
\end{align*}
Applying the representation with $g(x)=f(x)e^{x^2/2}$ gives the desired formula for the $b_n$, and moreover we have then the following relationship: 
\begin{align}\label{relatie} 
b_n=d_n\frac{(-1)^n}{\sqrt{2^n n! \sqrt{\pi}}}.
\end{align}
Now we want to approximate the derivatives of $e^{-x^2}$ in $L^2(\mathbb{R})$. This proposition is loosely based on Proposition 2 in \cite{calc}, which in which the errors are done in a different norm only for compactly supported functions. 
\begin{prop}\label{approx1}
For all $f\in L^2(\mathbb{R}^d,e^{x^2}\,dx)$, $N$ a fixed natural number and $\epsilon_0\in (0,1)$ we have that 
\begin{align*}
&\left(\int\limits_{\mathbb{R}}\left|f(x)-\sum\limits_{n=0}^Nc_ne^{-(x+n\epsilon_0)^2}\right|^2\,dx\right)^{\frac{1}{2}}\leq Ne^{N}\epsilon_0\|f\|_{L^2(\mathbb{R})}+E_N
\end{align*} 
where 
\begin{align}\label{cn}
c_n=\frac{1}{n!\sqrt{\pi}}\sum\limits_{k=n}^N\frac{(-1)^{n-k}}{(k-n)!(2\epsilon_0)^k}\int\limits_{\mathbb{R}}f(x)e^{x^2}\frac{d^k}{dx^k}e^{-x^2}\,dx
\end{align}
and 
\begin{align*}
E_N=\left(\sum\limits_{n=N+1}^{\infty}|d_n|^2\right)^{\frac{1}{2}} \qquad d_n=(-1)^n\int\limits_{\mathbb{R}}f(x)e^{x^2}\frac{d^n}{dx^n}e^{-x^2}\,dx.
\end{align*}
\end{prop} 

\begin{proof}[Proof of Proposition \ref{approx1}]
We start with setting
\begin{align*}
E_n(x)=\left(\frac{d^n}{dx^n}e^{-x^2}-\frac{1}{\epsilon_0^n}\sum\limits_{k=0}^n(-1)^{n-k}\binom{n}{k}e^{-(x+k\epsilon_0)^2}\right). 
\end{align*}
This is in $L^2(\mathbb{R})$ since each of the terms is in $L^2(\mathbb{R})$. By Taylor's theorem and the method of forward finite differences for $g\in C^{n+2}(\mathbb{R})$, $\epsilon_0\in(0,1)$ we have 
\begin{align*}
\frac{d^n}{dx^n}g(x)=\sum\limits_{k=0}^n (-1)^{n-k}\left(\binom{n}{k} \frac{1}{\epsilon_0^n}g(x+k\epsilon_0) -\frac{\epsilon_0^{n-k+1}}{(n+1)!}\binom{n}{k}k^{n+1}g^{n+1}(\xi_k)\right)
\end{align*}
where $\xi_k$ is between $x$ and $x+n\epsilon_0$. This is the Lagrange form of the remainder. If we set $g(x)=e^{-x^2}$ then this formula gives the difference $E_n(x)$.
We let $$\tilde{g}^{n+1}(x)=\frac{(-1)^n}{\sqrt{2^n n! \sqrt{\pi}}}\frac{d^{n+1}}{dx^{n+1}}e^{-x^2}.$$  
The point is that 
\begin{align*}
\sup_{\xi_k\in (x,x+n\epsilon_0)}|g^{n+1}(\xi_k)|^2=\sup_{l\in (0,n\epsilon_0)}|g^{n+1}(x+l)|^2=|g^{n+1}(x+l_*)|^2
\end{align*}
for some $l_*\in [0,n\epsilon_0]$. Let $r$ be a positive real number. Then by differentiability of $g^{n+1}(x)$, we have that there exists a sequence $l_j$ such that $|g^{n+1}(x+l_j)|^2\rightarrow |g^{n+1}(x+l_*)|^2$ uniformly on $\{x: |x|\leq r\}$.  It follows that 
\begin{align}\label{Nmin1}
&\int\limits_{|x|\leq r}\sup\limits_{\xi_k}|\tilde{g}^{n+1}(\xi_k)|^2\,dx=
\sup\limits_{l\in [0,\epsilon_0n]}\int\limits_{|x|\leq r}|\tilde{g}^{n+1}(x+l)|^2\,dx.
\end{align}
We then notice that
\begin{align*}
\int\limits_{\mathbb{R}}|H_n(x)|^2e^{-x^2}\,dx=\int\limits_{\mathbb{R}}e^{x^2}\left|\frac{d^{n}}{dx^{n}}e^{-x^2}\right|^2\,dx=2^nn!\sqrt{\pi}
\end{align*}
which implies since $\epsilon_0n>0$ and $e^{x^2}\geq 1$ for all $x\in\mathbb{R}$
\begin{align}\label{Nmin}
&\int\limits_{|x|\leq r}\sup\limits_{\xi_k}|\tilde{g}^{n+1}(\xi_k)|^2\,dx\leq \\& \nonumber
\sup\limits_{l\in [0,\epsilon_0n]}\int\limits_{\mathbb{R}}e^{(x+l)^2}|\tilde{g}^{n+1}(x+l)|^2\,dx \leq \int\limits_{\mathbb{R}}e^{x^2}|\tilde{g}^{n+1}(x)|^2\,dx\leq  2(n+1). 
\end{align}
Using Minkowski's inequality followed by Cauchy-Schwarz we have that 
\begin{align}\label{sums}
&\|\sum\limits_{n=0}^Nb_nE_n\|_{L^2(\mathbb{R})}=  \|\sum\limits_{n=0}^Nd_n\tilde{E}_n\|_{L^2(\mathbb{R})}\leq \sum\limits_{n=0}^N|d_n|\|\tilde{E}_n\|_{L^2(\mathbb{R})}\leq \\&  \left(\sum\limits_{n=0}^N|d_n|^2\right)^{\frac{1}{2}}\left(\sum\limits_{n=0}^N\|\tilde{E}_n\|^2_{L^2(\mathbb{R})}\right)^{\frac{1}{2}}. \nonumber
\end{align} 
We then have 
\begin{align}\label{est}
&\epsilon_0^{-2}\|\tilde{E}_n\|_{L^2(|x|\leq r)}^2\leq \nonumber 
\int\limits_{|x|\leq r}\left|\sum\limits_{k=0}^n\frac{k^{k+1}\binom{n}{k}}{(n+1)!}\sup\limits_{\xi_k}|\tilde{g}^{n+1}(\xi_k)| \right|^2\,dx \\&\leq \sum\limits_{k=0}^n\left| \frac{k^{k+1}\binom{n}{k}}{(n+1)!}\right |^2\sum\limits_{m=0}^n\left(\int\limits_{|x|\leq r}\sup\limits_{\xi_m}|\tilde{g}^{n+1}(\xi_m)|^2\,dx\right)\\& \leq  \nonumber
\sum\limits_{k=0}^n\left|\frac{k^{k+1}}{k!(n-k)!}\right|^2\leq e^{2n}n
\end{align}  
where in the last line we have used estimate \eqref{Nmin}. 
Combining \eqref{sums} and \eqref{est} we have by the monotone convergence theorem
\begin{align*}
&\|\sum\limits_{n=0}^Nb_nE_n\|_{L^2(\mathbb{R})}\leq \epsilon_0Ne^{N}\|f\|_{L^2(\mathbb{R})}.
\end{align*}
It remains to bound
\begin{align*}
&\|\sum\limits_{n=N+1}^{\infty}b_n\frac{d^n}{dx^n}e^{-x^2}\|_{L^2(\mathbb{R})}.
\end{align*}
By orthogonality of the $h_n's$, 
\begin{align*}
&\|\sum\limits_{n=N+1}^{\infty}b_n\frac{d^n}{dx^n}e^{-x^2}\|_{L^2(\mathbb{R},e^{x^2}\,dx)}\leq \left(\sum\limits_{n=N+1}^{\infty}|d_n|^2\right)^{\frac{1}{2}}.
\end{align*}
In order to obtain the coefficients $c_n$ we note that 
\begin{align*}
\sum\limits_{n=0}^Nb_n\epsilon_0^{-n}\sum\limits_{k=0}^n(-1)^{n-k}\binom{n}{k}e^{-(x+k\epsilon_0)^2}=\sum\limits_{k=0}^N\left(\sum\limits_{n=k}^N(-1)^{n-k}\binom{n}{k}b_n\epsilon_0^{-n}\right)e^{-(x+k\epsilon_0)^2}
\end{align*}
which is Iverson's summation technique (c.f. equation 2.32 in Section 2.4 of \cite{knuth} p.36). 
\end{proof}

\begin{proof}[Proof of Theorem \ref{approx2}]
The proof follows using the product topology in the previous definition. We define a generalized Hermite function as 
\begin{align*}
h_m(x)=\prod\limits_{i=1}^d h_{m_i}(x_i)
\end{align*}
where $m=(m_1, . . . ,m_d)$ is now a multi-index of degree $d$ and $h_{m_i}(x_i)$ is the one dimensional Hermite function. The proof follows using the product topology from the previous Proposition. Indeed if 
\begin{align*}
f(x)e^{|x|^2/2}=\sum\limits_{m=0}^\infty \alpha_m \prod\limits_{i=1}^d h_{m_i}(x_i)
\end{align*}
with 
\begin{align*}
\alpha_m=\int\limits_{\mathbb{R}^d}f(x)e^{\frac{|x|^2}{2}}\prod\limits_{i=1}^dh_{m_i}(x_i)\,dx
\end{align*}
Then we can define a $\beta_m$ (d-dimensional version of $b_m$) corresponding to $\alpha_m$ ($d_m$ as in the 1-dimensional case), as in \eqref{relatie}
\begin{align*}
\beta_m=\prod_{i=1}^d\left(\frac{(-1)^{m_i}}{\sqrt{m_i! 2^{m_i}\sqrt{\pi}}}\int\limits_{\mathbb{R}^d}f(x)e^{\frac{|x|^2}{2}}h_{m_i}(x_i)\,dx \right).
\end{align*}
We can write almost as before: 
\begin{align*}
\frac{d^n}{dx^n}e^{x^2}=\prod_{i=1}^d \left(\sum\limits_{k_i=0}^{n_i}(-1)^{n_i-k_i}\binom{n_i}{k_i} \frac{1}{\epsilon_0^{n_i}} \right)e^{-|x+n\epsilon_0|^2}+\mathcal{O}(\epsilon_0^d\left(e^{2Nd}N^d\right)^{\frac{1}{2}})
\end{align*}
with the equality in the sense of $L^2(\mathbb{R}^d)$. 
If $f(x)=\prod\limits_{i=1}^df_i(x_i)$ then $\beta_n=\prod\limits_{i=1}^d\beta_{n_i}$. This gives 
\begin{align*}
&\sum\limits_{|n|\leq N}\beta_n\prod\limits_{i=1}^d\left( \sum\limits_{k_i=0}^{n_i}(-1)^{n_i-k_i}\binom{n_i}{k_i} \frac{1}{\epsilon_0^{n_i}} \right)e^{-|x+n\epsilon_0|^2}= \\& \sum\limits_{|n|\leq N}\prod\limits_{i=1}^d\left(\beta_{n_i}\sum\limits_{k_i=0}^{n_i}(-1)^{n_i-k_i}\binom{n_i}{k_i} \frac{1}{\epsilon_0^{n_i}} \right)e^{-|x+n\epsilon_0|^2}=\\& \prod\limits_{i=1}^d\sum\limits_{n_i=0}^{N_i}\left(\beta_{n_i}\sum\limits_{k_i=0}^{n_i}(-1)^{n_i-k_i}\binom{n_i}{k_i} \frac{1}{\epsilon_0^{n_i}} \right)e^{-|x+n\epsilon_0|^2}.
\end{align*} 
Again using Iverson's summation technique on each of the coordinates/sums $i=1, . . .,d$ seperately we have that
\begin{align*}
&\sum\limits_{n_i=0}^{N_i}\left(\beta_{n_i}\sum\limits_{k_i=0}^{n_i}(-1)^{n_i-k_i}\binom{n_i}{k_i} \frac{1}{\epsilon_0^{n_i}} \right)e^{-|x+n_i\epsilon_0|^2}= \sum\limits_{n_i=0}^{N_i}\left(\sum\limits_{k_i=n_i}^{N_i}(-1)^{k_i-n_i}\binom{k_i}{n_i} \frac{\beta_{k_i}}{\epsilon_0^{k_i}} \right)e^{-|x+n_i\epsilon_0|^2}.
\end{align*}
We then obtain 
\begin{align*}
&\prod\limits_{i=1}^d\sum\limits_{n_i=0}^{N_i}\left(\sum\limits_{k_i=n_i}^{N_i}(-1)^{k_i-n_i}\binom{k_i}{n_i} \frac{\beta_{k_i}}{\epsilon_0^{k_i}} \right)e^{-|x+n\epsilon_0|^2}=\\& 
\sum\limits_{|n|\leq N}\left(\sum\limits_{k_1=n_1}^{N_1}(-1)^{k_1-n_1}\binom{k_1}{n_1} \frac{\beta_{k_1}}{\epsilon_0^{k_1}}. . . .\sum\limits_{k_d=n_d}^{N_d}(-1)^{k_d-n_d}\binom{k_d}{n_d} \frac{\beta_{k_d}}{\epsilon_0^{k_d}} \right)e^{-|x+n\epsilon_0|^2}=\\&
\sum\limits_{|n|\leq N}\left(\left(\sum\limits_{k_1=n_1}^{N_1}(-1)^{k_1-n_1}\binom{k_1}{n_1} \frac{1}{\epsilon_0^{k_1}}. . . .\sum\limits_{k_d=n_d}^{N_d}(-1)^{k_d-n_d}\binom{k_d}{n_d} \frac{1}{\epsilon_0^{k_d}}\right)\beta_k \right)e^{-|x+n\epsilon_0|^2}
\end{align*}
which defines the coefficients $c_n$, \eqref{cn}, in the $d$ dimensional case. Here we have used the fact that $f(x)$ can be approximated by products of one dimensional functions and taken the limit. 
\end{proof}
We have the bound on the errors for compactly supported functions with sufficient regularity. We define $H^j([-M,M])$ as the space of $H^j(\mathbb{R}^d)$ functions having compact support in $[-M,M]^d$. 
\begin{prop}\label{compact1}
We have for $f\in H^3([-M,M]^d)$, $N>2$
\begin{align*}
E_N\leq \frac{p_3(M)||f||_{H^3(\mathbb{R}^d,e^{x^2}\,dx)}}{N^{\frac{d}{4}}}
\end{align*}
with $p_3(M)$ denoting a polynomial of degree $3$ in $M^d$.  
\end{prop}
\begin{proof}
We show the result in dimension 1 with the obvious generalisation to dimension $d$. Loosely inspired by \cite{boyd}, if generically $f_1\in H^j([-M,M])$, $m\in \mathbb{N}$ then the bound on the corresponding $d_n$ is given to us by 
\begin{align*}
|d_n|\leq \frac{\|\left(x+\frac{d}{dx}\right)^jf_1\|_{L^1(\mathbb{R})}}{(2(n+1))^{\frac{j}{2}}}.
\end{align*}
This is a consequence of the fact
\begin{align*}
\sqrt{2(n+1)}h_{n+1}(x)=x h_n(x)-\frac{d}{dx}h_n(x)
\end{align*}
and integration by parts. This bound implies for $f_1=fe^{\frac{x^2}{2}}$, $f\in H^j([-M,M])$
\begin{align*}
\left(\sum\limits_{N+1}^{\infty}|d_n|^2\right)^{\frac{1}{2}}=E_N
\leq \|\left(x+\frac{d}{dx}\right)^j(fe^{\frac{x^2}{2}})\|_{L^1(\mathbb{R})}\left(\sum\limits_{n=N+1}^{\infty}\frac{1}{(2(n+1))^{\frac{j}{2}}}\right)^{1/2}.
\end{align*}
We set $j=3$ and want to estimate the sum on the right hand side using an integral. We recall the Euler summation formula which is eq. 9.67, Section 9.5 p 455 of \cite{knuth}. Given an integer valued function $h(\alpha)$, $h(\alpha)$ may be estimated by the Euler summation formula
\begin{equation}\label{Euler}
\sum\limits_{a\leq \alpha\leq b}h(\alpha)=\int\limits_{a}^{b}h(x)\,dx +\sum\limits_{m=1}^{k}\frac{B_m}{m!}h^{m-1}(x)|_{x=a}^{x=b} +R_k
\end{equation}
where $B_m$ is the $m^{th}$ Bernoulli number and $h^{m}(x)$ denotes the $m^{th}$ derivative of $h(x)$.  The remainder $R_k$ is defined as 
\[
R_{k}=(-1)^{k+1}\int\limits_{\mathbb{R}}\frac{B_k(\{x\})}{k!}h^k(x)\,dx.
\]
The notation $\{x\}$ denotes the fractional part of $x$, and $B_k(\{x\})$ denotes the $k^{th}$ Bernoulli polynomial. Applying this gives for $N>2$
\begin{align*}
\left(\sum\limits_{n=N+1}^{\infty}\frac{1}{(2(n+1))^{\frac{3}{2}}}\right)^{1/2}\leq \frac{10}{N^{\frac{1}{4}}}
\end{align*} 
were we have used the fact that $|B_2(\{x\})|\leq B_2=1/2$ for all $x\in \mathbb{R}$. 
\end{proof}
\begin{cor}\label{candidate}
Define for $b_k$ real numbers indexed by $k=(k_1, . . .,k_d)$, $k_i\in \{1, . . .,N\}$
\begin{align*}
\psi(x)= \sum\limits_{|k|\leq N}b_{k}\left( \prod\limits_{i=1}^dh_{k_i}(x_i)\right)e^{-\frac{|x|^2}{2}}
\end{align*}
Then $\psi$ is such that $\eta$ in \eqref{etadef} has $E_N=0$.  
\end{cor}
\begin{proof}
The coefficients $d_n=0$ for all $|n|\geq N+1$ in the case $d=1$ with this choice of $\psi$. The result then follows immediately since the same statement holds for each coordinate $x_i$. 
\end{proof}
Unfortunately because the approximation method is based on the method of forward finite differences the coefficients of the Gaussians $c_n$ alternate and only very specific sums of the form 
\begin{align*}
\sum\limits_{|k|\leq N}\epsilon_0^kb_kh_k(x)e^{\frac{-|x|^2}{2}}
\end{align*}
have positive coefficients $\tilde{c}_n$ in their approximations and are therefore in the class $\mathcal{A}$. For functions which have been extended from characteristic functions to piecewise $C^1$ compactly supported functions are better off decomposed into Gaussians in a natural way based on Riemann integration in Lemma \ref{stepg} below. 

\begin{lem}[Example of functions in $\mathcal{A}$]\label{stepg}
There exist piecewise $C^1$ compactly supported functions which agree with a characteristic function $\chi_B$ on $B_{M/2}$ a ball of radius $M/2$, for all $M\geq 2$ which are in $\mathcal{A}$ for some $\eta\in (0,1)$. 
\end{lem}
\begin{proof}
We show this in $1d$ with the obvious generalisation holding. Let $M$ be a real positive number, and $\Delta x\in (0,1)$. Loosely, following \cite{geophysics} in 1d we create a Riemann sum
\begin{align*}
\sum\limits_{n=0}^N\frac{\Delta x_n}{\sqrt{\pi}}e^{-(x-x_n+M)^2} \quad |x_n|\leq 2M \quad x_n=0+n\Delta x
\end{align*}
whose coefficients $\Delta x_n$ are all positive. Using the fact that $\mathrm{erf}(x)$ is monotone decreasing in $x$, by the standard error estimates in Riemann integration
\begin{align}\label{Rerror}
\sup\limits_{x\in\mathbb{R}}\left|\sum\limits_n  \frac{\Delta x_n}{\sqrt{\pi}}e^{-(x-x_n+M)^2}-\frac{1}{\sqrt{\pi}}\int\limits_{0}^{2M}e^{-|x-y+M|^2}\,dy\right|<2M\Delta x.
\end{align} 
We take this further. Let $\chi_{B_{M/2}}$ be the characteristic function on $[-M/2,M/2]$. Then changing variables gives 
\begin{align*}
\sup\limits_{x\in [-M/2,M/2]}\left|\chi_{B_{M/2}}-\frac{1}{\sqrt{\pi}}\int\limits_{-M}^Me^{-|x-y|^2}\,dy\right|\leq 2\mathrm{erfc}(M/2).
\end{align*}
We extend the characteristic function 
\begin{align*}
&\phi_1(x)=\\& \nonumber
\chi_{B_{M/2}}(x)+\chi_{[M/2,\infty)}(x)\frac{\frac{1}{\sqrt{\pi}}\int\limits_{-M}^Me^{-|x-y|^2}\,dy}{\frac{1}{\sqrt{\pi}}\int\limits_{-M}^Me^{-|M/2-y|^2}\,dy}+\chi_{(-\infty,-M/2]}(x)\frac{\frac{1}{\sqrt{\pi}}\int\limits_{-M}^Me^{-|x-y|^2}\,dy}{\frac{1}{\sqrt{\pi}}\int\limits_{-M}^Me^{-|-M/2+y|^2}\,dy}.
\end{align*}
It follows that 
\begin{align*}
\sup\limits_{x\in \mathbb{R}}\left|\frac{1}{\sqrt{\pi}}\int\limits_{-M}^Me^{-|x-y|^2} -\phi_1\right|\leq \mathrm{erfc}(M/2)+\mathrm{erfc}(3M/2)
\end{align*}
because for $x\in [M/2,\infty)$
\begin{align*}
2\left|\frac{\frac{1}{\sqrt{\pi}}\int\limits_{-M}^Me^{-|x-y|^2}\,dy}{\frac{1}{\sqrt{\pi}}\int\limits_{-M}^Me^{-|M/2-y|^2}\,dy}-\frac{1}{\sqrt{\pi}}\int\limits_{-M}^Me^{-|x-y|^2}\,dy\right|\leq \mathrm{erfc}(M/2)+\mathrm{erfc}(3M/2)
\end{align*} 
and similarly for the other integral. Here we have used the fact 
\begin{align*}
\mathrm{erfc}(1)\leq 0.16 \quad \mathrm{erfc}(3)\leq 0.01
\end{align*}
and for $M\geq 2$
\begin{align*}
\left|\frac{1}{\frac{1}{\sqrt{\pi}}\int\limits_{-M}^Me^{-|M/2-y|^2}\,dy}-1\right| \leq \frac{\mathrm{erfc}(1)+\mathrm{erfc}(3)}{2-\mathrm{erfc}(1)-\mathrm{erfc}(3)}.
\end{align*}
If we want the extension to be compactly supported we define
\begin{align*}
&\phi(x)=\\& \nonumber
\chi_{B_{M/2}}(x)+\varphi_{[M/2,10M)}(x)\frac{\frac{1}{\sqrt{\pi}}\int\limits_{-M}^Me^{-|x-y|^2}\,dy}{\frac{1}{\sqrt{\pi}}\int\limits_{-M}^Me^{-|M/2-y|^2}\,dy}+\varphi_{(-10M,-M/2]}(x)\frac{\frac{1}{\sqrt{\pi}}\int\limits_{-M}^Me^{-|x-y|^2}\,dy}{\frac{1}{\sqrt{\pi}}\int\limits_{-M}^Me^{-|-M/2+y|^2}\,dy}
\end{align*}
where $\varphi$ is $1$ on $[M/2,9M)$, smooth on $[9M,10M]$ and $0$ elsewhere. 
We see that for $M\geq 2$
\begin{align*}
\sup\limits_{x\in \mathbb{R}}\left|\phi -\phi_1\right|\leq \mathrm{erfc}(9M).
\end{align*}
It follows that 
\begin{align*}
&\sup\limits_{x\in \mathbb{R}}\left|\phi -\sum\limits_n\frac{\Delta x_n}{\sqrt{\pi}}e^{-(x-x_n-M)^2}\right| \\&
\leq \mathrm{erfc}(9M)+\mathrm{erfc}(M/2)+2\Delta x M\leq 2e^{-M^2/4}+2\Delta x M.
\end{align*}
Thus $\phi(x)$ is the desired extension of $\chi_B$ if $M$ is sufficiently large and $\Delta x<4M$. We can also shift the function further away from the origin if we want all the $x_n$ positive or negative. 
\end{proof}
If we want to rescale the ball to $M\epsilon$ then we use $\chi_{B_{M/2}}(\frac{x}{\epsilon})$ with $\epsilon\in (0,1)$ and the Gaussians are also rescaled. The proof of Theorem \ref{main3} still holds with minor modifications under this rescaling. Notice that for the $d$-dimensional extension $\|\phi\|_{L^2(\mathbb{R}^d)}=\mathcal{O}(M^d)$, so this error $\eta$ in \eqref{obs} is relatively quite small compared to the $L^2$ norm of $\phi$. We show in Section \ref{conclusions} that this $\phi$ is $L^2\textendash L^{\infty}$ observable. 

\section{Representations of Solutions in terms of Gaussian Wavepackets, Proof of Theorem \ref{parametrix}}\label{wavepacketconstruction}
We would like to represent solutions to \eqref{nonauto} in terms of Gaussians. To do this, we decompose functions in $L^2(\mathbb{R}^d)$ as the sum of a finite number of real gaussians. Then we will apply the FIO from the previous section. 
\begin{lem}\label{gaussianwavepackets}
The solution to \eqref{nonauto} with $a_n$ in $\mathbb{R}^d$
\begin{align*}
u_0=\phi_n(x)=\phi_n(0,x)=e^{-|x+a_n|^2}, 
\end{align*}
can be written as 
\begin{align*}
u(t,x)=\phi_n(t,x)=\left(\frac{a^2(t)}{(1-4iy_3(t))}\right)^{\frac{d}{2}}e^{iy_1(t)|x|^2-\frac{|y_2(t)x+a_n|^2}{(1-4iy_3(t))}}.
\end{align*} 
\end{lem} 
\begin{proof}
Recall that 
\begin{align*}
\int\limits_{\mathbb{R}^d}e^{-iy\cdot \eta}e^{-c|y|^2}\,dy=\left(\frac{\pi}{c}\right)^{\frac{d}{2}}e^{-\frac{|\eta|^2}{4c}}.  
\end{align*}
The Fourier transform of $\phi_n(x)$ is 
\begin{align*}
2^{-\frac{d}{2}}e^{ia_n\cdot\eta-\frac{|\eta|^2}{4}}.
\end{align*}
The FIO applied to $\phi_n(x)$ is then 
\begin{align*}
&\left(\frac{a^2(t)}{4\pi}\right)^{\frac{d}{2}}\int\limits_{\mathbb{R}^d}e^{iy_1(t)|x|^2+iy_2(t)x\cdot\eta+iy_3(t)|\eta|^2+ia_n\cdot\eta-\frac{|\eta|^2}{4}} \,d\eta\\& \nonumber
=\left(\frac{a^2(t)}{(1-4iy_3(t))}\right)^{\frac{d}{2}}e^{iy_1(t)|x|^2-\frac{|y_2(t)x+a_n|^2}{(1-4iy_3(t))}}
\end{align*} 
which gives the desired result. 
\end{proof} 
We can then construct a nearly explicit parametrix to all $u_0=f(x)e^{-|x|^2/2}$ with $f(x)$ an $L^2(\mathbb{R}^d)$ function which is interesting in its own right
\begin{proof}[Proof of Theorem \ref{parametrix}]
The approximation in \eqref{approxH} holds for any $f\in L^2(\mathbb{R}^d)$ because $h_k(x)$ is a well known (complete) orthonormal basis for $L^2(\mathbb{R}^d)$. More explicitly we see, \eqref{approxH} holds for a fixed $\eta$ and $f\in H^3([-M,M]^d)$ since we can select $N$ larger than $M$ based on the threshold $\eta$, using Proposition \ref{compact1}. We then use Corollary \ref{candidate} to give the two term bound on the error approximating $u_0$. The principle of superposition and the construction of the individual $\phi_n(t,x)$ in Lemma \ref{gaussianwavepackets} gives the desired result. The error over time can be bounded since the non-autonomous propagation operator is still an isometry on $L^2(\mathbb{R}^d)$ as a result of \cite{karel}. 
\end{proof}

\section{Proof of Theorems \ref{main} and \ref{main3} for Observability} 
From Theorem \ref{approx2} and Lemma \ref{gaussianwavepackets}, we are able to construct accurate solutions to \eqref{nonauto} using sophisticated Gaussian wavepackets. Gaussian wavepackets have localization features which will allow us to analyze the observable sets, $\omega$. Therefore, we start by proving some elementary bounds on the solutions to \eqref{system}, provided it is well-posed, which will then allow us to analyze the observability of the solution. A future goal is the development of complex valued initial data into wavepackets which would pose the added challenge that there could be a wavefront set, \cite{rodino}. 
\begin{lem}\label{Glemma}
Assume the Hamiltonian flow is non-zero for all $0<t\leq T_D$. Let $K_0,K$ be positive constants such that for all $t\in [0,T_D]$
\begin{align*}
|y_1(t)|\leq K_0 \quad \quad
0< \kappa_1(t)\leq K.
\end{align*}
Then the solutions $y_2(t)$ and $y_3(t)$ to \eqref{system} satisfy the following bounds
\begin{align}\label{Gbounds}
&e^{-4K_0Kt} \leq y_2(t)\leq e^{4K_0Kt}\quad \quad y_3(t)\leq 0.
\end{align}
\end{lem}
\begin{proof}
The proof follows immediately from Gronwall's inequality and \eqref{system}. 
\end{proof}
In practice we will compute the bounds explicitly in Section \ref{apps} when we know that the Hamiltonian flow is non-zero. Before proving a lower bound, we recall the following result of \cite{lowerbound}. The error function is defined by
\begin{align*}
\mathrm{erfc}(x)=\frac{2}{\sqrt{\pi}}\int\limits_x^{\infty} e^{-x^2}\,dx.
\end{align*}
We have the following single term lower bound for the complementary error function
\begin{align}\label{lowerbd}
\mathrm{erfc}(x)\geq \sqrt{\frac{2e}{\pi}}\sqrt{\frac{\beta-1}{\beta}}e^{-\beta x^2} \quad x\geq 0,\,\, \beta>1.
\end{align}
For the rest of this paper we make the definitions  
\begin{align}\label{adef}
A=A(t)=\frac{2|y_2(t)|^2}{(1+16y_3^2(t))}  \quad \gamma(t)=\left(\frac{a^2(t)}{(1-4iy_3(t))}\right)^{\frac{d}{2}}&\\ \nonumber
\end{align}
which will control the spread of the wave packets and ultimately the constant $C_T$. 
The above results will be useful in proving the lower bound in the next lemma.
\begin{lem}\label{lowerinner}
Let $\phi_n(t,x)$ be the same as in Lemma \ref{gaussianwavepackets} for arbitrary $a_n\in \mathbb{R}^+$. Let $c_n$ be a finite collection of  real numbers indexed by $n\in \mathbb{R}^d$, $|n|\leq N$, then have the following lower bound for all $0\leq t\leq T$
\begin{align*}
\epsilon(t,R)\sum\limits_{|n|\leq N}|c_n|\leq  \sum\limits_{|n|\leq N}\left(\int\limits_{[-R,R]^d}|c_n\phi_n(t,x)|^2\,dx\right)^{\frac{1}{2}}
\end{align*}
where $\epsilon(t,R)$ is a positive constant. 
\end{lem}
\begin{proof}
Define 
\begin{align*}
 \quad\quad B_n= \frac{a_n}{y_2(t)}
 \end{align*}
then we have that
\begin{align*}
&\int\limits_{([-R,R]^d)^c}|\phi_n(t,x)|^2\,dx
\geq |\gamma(t)|^2 \int\limits_{([-R,R]^d)^c}e^{-\frac{2|y_2(t)x+a_n|^2}{(1+16y_3^2(t))}}\,dx\\&=|\gamma(t)|^2 \left(\frac{\pi}{4A}\right)^{\frac{d}{2}}\prod\limits_{n_i}(h(|B_{n_i}|,A,R))
\end{align*}
where
\begin{align*}
&h(|B_{n_i}|,A,R)\\&=
\begin{cases}
\mathrm{erfc}\left( \sqrt{A}\left(R+\left|B_{n_i}\right|\right)\right)+ \mathrm{erfc}\left(\sqrt{A}\left(R-\left|B_{n_i}\right|\right)\right) & R>\left|B_{n_i}\right| \\ 
\mathrm{erfc}\left(\sqrt{A}\left(R+\left|B_{n_i}\right|\right)\right)+ \left(2-\mathrm{erfc}\left(\sqrt{A}\left(\left|B_{n_i}\right|-R\right)\right)\right)  & R\leq \left|B_{n_i}\right|. 
\end{cases}
\end{align*} 
If $R>\left|B_{n_i}\right|=m\geq 0$, then from equation \eqref{lowerbd} we have with $\beta=2$ 
\begin{align*}
&\mathrm{erfc}\left(\sqrt{A}(R-m)\right)\geq \sqrt{\frac{e}{\pi}}e^{-2A(R-m)^2} \geq \sqrt{\frac{e}{\pi}}e^{-2AR^2} .
\end{align*} 
If $\left|B_{n_i}\right|\geq R\geq 0$ then 
\begin{align*}
&\left(2-\mathrm{erfc}(\sqrt{A}(m-R))\right)>1.
\end{align*} 
As a result we obtain 
\begin{align}
&\epsilon(t,R)=\left(\frac{\pi}{4A}\right)^{\frac{d}{4}}|\gamma(t)|\min\left\{1, \sqrt[4]{\frac{e}{\pi}}e^{-AR^2}\right\}.
\end{align}
Noticing that $|\gamma(t)|=\left(\frac{A}{2}\right)^{\frac{d}{4}}$ we can simplify further
\begin{align}\label{defepst}
&\epsilon(t,R)=\min\left\{\frac{\pi^{\frac{d-1}{4}}e^{\frac{1}{4}}}{2^{\frac{3d}{4}}}e^{-AR^2},\left(\frac{\pi}{8}\right)^{\frac{d}{4}}\right\}.
\end{align}
By combining Lemmata \ref{y1} and \ref{Glemma}, $A$ is bounded above and below for all $0\leq t\leq T$ depending on the $L^{\infty}([0,T])$ norms of $\kappa_1(t),\kappa_2(t)$, implying $\epsilon(t,R)$ is positive. An explicit lower bound for \eqref{defepst} is computed for various examples in Section \ref{apps}. 
\end{proof}
Recall the Diaz-Metcalf inequality
\begin{lem}\label{reversecomplex}[Diaz-Metcalf \cite{diazmetcalf}]
Let $w$ be a unit vector in the inner product space $(H;\langle \cdot, \cdot\rangle)$ over the real or complex number field. Suppose that the vectors $v_n\in H\setminus{0}$, $i\in\{1, . . ,N\}$ satisfy
\begin{align}
0\leq r\leq \frac{\mathrm{Re}\langle v_n,w\rangle}{\|v_n\|}
\end{align}
then  
\begin{align}
r\sum\limits_{i=1}^N\|v_n\|\leq \|\sum\limits_{i=1}^Nv_n\|. 
\end{align}
\end{lem} 
Using the above lemma we have can find a lower bound on the $L^2((0,T)\times \omega)$ norm of $u(t,x)$ in terms of the Gaussian wave packets which suits our needs. The $L^2((0,T)\times \omega)$ norm is difficult to work with directly when the initial data is propagated into complex wavepackets, because their inner products take the form of Frensel integrals. This is the reason for the intermediate Lemmata. We make the following choice of $a_n$. 
\begin{lem}\label{ass1}
There is a choice of $\varepsilon$ in terms of a fixed natural number $N$ such that for $a_n=\varepsilon n$ with $n=(n_1, . . .,n_d)$, $|n|\leq N$ $n_i\in \{1, 2, . . .N\}$ and a $\psi_1(x)$ independent of $n$ such that 
\begin{align*}
\mathrm{Re}\int\limits_{|x|>R_0}(\gamma(t))^{-1}\phi_n(t,x)\overline{\psi_1(x)}\,dx\geq \tilde{\delta}(t,R_0)>0 
\end{align*}
for all of these $n$, where $B_{R_0}$ is the ball of radius $R_0$ with $||\psi_1||_{L^2(B_{R_0}^c)}=1$.
\end{lem}
\begin{proof}
Select 
\begin{align*}
\psi_0(x)=\frac{e^{-\frac{A}{2}|x|^2-2iy_3(t)A|x|^2+iy_1(t)|x|^2}}{|x|^{d-1}} 
\end{align*}
then we have that 
\begin{align*}
\int\limits_{|x|>R_0}|\psi_0(x)|^2\,dx=\int\limits_{|x|>R_0}\frac{e^{-A|x|^2}}{|x|^{2(d-1)}}\,dx=(C(R_0,A))^2
\end{align*}
where $C(R_0,A)=\mathcal{O}(e^{-R_0^2A/2})$. We set $\psi_0=(C(R_0,A))\psi_1$. 
It follows that 
\begin{align*}
&\int\limits_{|x|>R_0}\phi_n(t,x)\overline{\psi_0(x)}\,dx=\int\limits_{|x|>R_0}\frac{e^{-A|x|^2+\tilde{A}y_2(t)a_n\cdot x-\frac{\tilde{A}}{2}|a_n|^2}}{|x|^{d-1}}\,dx\\&=
\int\limits_{\theta\in \mathbb{S}^{d-1}} \int\limits_{r>R_0}e^{-Ar^2+\tilde{A}y_2(t)a_n\cdot b(\theta)r-\frac{\tilde{A}}{2}a_n^2}\,dr\,d\theta.
\end{align*}
Here $\tilde{A}=2(1-4iy_3(t))^{-1}$ and $b(\theta)$ is a vector depending on $\theta\in \mathbb{S}^{d-1}$ with norm 1. It follows that if $a_n=\varepsilon n$, where we have control over the parameter $\varepsilon$ which is small, then we can Taylor expand this integral in terms of $\varepsilon$. 
Set 
\begin{align*}
\int\limits_{|x|>R_0}\frac{e^{-A|x|^2+\tilde{A}y_2(t)a_n\cdot x-\frac{\tilde{A}}{2}|a_n|^2}}{|x|^{d-1}}\,dx=f(\varepsilon,A,R_0,n).
\end{align*}
Let $C(d)$ be the volume of the unit ball in $d$ dimensions. The leading order term is
\begin{align*}
f(0,A,R_0,n)= C(d)\int\limits_{r>R_0}e^{-Ar^2}\,dr,
\end{align*}
which is independent of $n$. The error in this approximation is bounded by 
\begin{align*}
&\left|\sup\limits_{\varepsilon\in(0,1)}\frac{d}{d\varepsilon}f(\varepsilon,A,R_0,n)\right|\leq \\&\sup\limits_{\varepsilon\in(0,1)}\int\limits_{r>R_0}C(d)|\tilde{A}n|(1+r|y_2(t)|)e^{-Ar^2+A_1|n|r}\,dr=C(d)nC_2(A_1,R_0,n).
\end{align*}
with $A_1=2|y_2(t)(1+16y_3^2(t))^{-1}|$. For positive constants $a,b,c,m$ we have that 
\begin{align}
\int\limits_m^{\infty}(a+bx)e^{cx-x^2}\,dx=
\frac{e^{\frac{c^2}{4}}}{4}\left(\sqrt{\pi}(2a+bc)\left(\mathrm{erf}\left(\frac{1}{2}(c-2m)\right)+1\right)+2be^{-\frac{1}{4}(c-2m)^2}\right).
\end{align}
Using the substitution 
\begin{align*}
a=1\quad b=\frac{|y_2(t)|}{\sqrt{A}} \quad c=\frac{A_1n}{\sqrt{A}} \quad \sqrt{A}R_0=m,
\end{align*}
we find
\begin{align*}
&\sup\limits_{n\leq N}C_2(A_1,R_0,n)\leq \\&
 \frac{C(d)|\tilde{A}|}{2\sqrt{A}}\left(\frac{|y_2(t)|}{\sqrt{A}}e^{-AR_0^2+A_1NR_0}+2\sqrt{\pi}e^{\frac{A_1^2N^2}{A}}\left(1+\frac{A_1N}{A}|y_2(t)|\right)\right).
\end{align*}
Plugging in the values of $A$, $\tilde{A}$ and $A_1$ in terms of the Riccati equations we obtain
\begin{align}\label{normest}
&C(d)^{-1}\sup\limits_{n\leq N}C_2(A_1,R_0,n)\leq  \\& \nonumber
\frac{\sqrt{1+16y_3(t)^2}}{2|y_2(t)|}e^{\frac{2|y_2(t)|^2R_0^2+2|y_2(t)|NR_0}{1+16y_3^2(t)}}+\sqrt{2\pi}\left(\frac{N+1}{|y_2(t)|} \right)e^{\frac{2N^2}{1+16y_3^2(t)}}.
\end{align}
 This gives   
\begin{align*}
\frac{1}{C(R_0,A)}\sup\limits_{n\leq N}\Re \int\limits_{|x|>R_0}\phi_n(t,x)\overline{\psi_0(x)}\,dx=C(d)\frac{\int\limits_{r>R_0}e^{-Ar^2}\,dr}{C(R_0,A)}+\mathcal{O}(\varepsilon)=\tilde{\delta}(t,R_0)>0,
\end{align*}
by our choice of $\varepsilon$ sufficiently small. Notice that 
\begin{align*}
C(d)\frac{\int\limits_{r>R_0}e^{-Ar^2}\,dr}{C(R_0,A,n)}\geq R_0^{\frac{d-1}{2}}\left(\frac{\mathrm{erfc}(\sqrt{A}R_0)}{2\sqrt{A}}\right)^{\frac{1}{2}}
\end{align*} 
for $R_0>1$. We need to show that the $\mathcal{O}(\varepsilon)$ terms are not so large so that $\tilde{\delta}(t,R_0)$ is positive. Using Taylor's theorem we see that this can be accomplished by using \eqref{normest} to enforce the condition
\begin{align}\label{req}
&\frac{\sqrt{1+16y_3(t)^2}}{2|y_2(t)|}e^{\frac{2|y_2(t)|^2R_0^2+2|y_2(t)|NR_0}{1+16y_3^2(t)}}+\sqrt{2\pi}\left(\frac{N+1}{|y_2(t)|} \right)e^{\frac{2N^2}{1+16y_3^2(t)}}
\leq \\& \nonumber \frac{1}{2}\int\limits_{r>R_0}e^{-Ar^2}\,dr.
\end{align}
 We have for $\varepsilon$ sufficiently small that 
 \begin{align}\label{deltatilde}
 2\tilde{\delta}(t,R_0)>R_0^{\frac{d-1}{2}}\left(\frac{\mathrm{erfc}(\sqrt{A}R_0)}{2\sqrt{A}}\right)^{\frac{1}{2}}>R_0^{\frac{d-1}{2}}\left(\frac{e}{4A\pi}\right)^{\frac{1}{4}}e^{-AR_0^2}.
 \end{align} 
 This finishes the proof. In practice the constants in time in \eqref{deltatilde} and \eqref{req} are computable and this is done in Section \ref{apps} for various examples. 
\end{proof}
 Using these lemmata we can show that in essence it is possible to ignore some of the problems caused by the off diagonal terms for certain $f(x)$. Let $R_0$ and $R$ be large enough such that $\Omega\subset B_{R_0}\subset [-R,R]^d$. 
\begin{lem}\label{chainlemma}
Let
\begin{align}\label{delta}
\delta(t,R_0)=\frac{e^{\frac{1}{4}}A^{\frac{d-1}{4}}R_0^{\frac{d-1}{2}}}{2^{\frac{d}{4}+1}(4\pi)^{\frac{1}{4}}}e^{-AR_0^2}
\end{align}
Assume that $a_n$ satisfies Lemma \ref{ass1}, and $c_n$ are finite real positive numbers then we have the following inequalities
\begin{align}\label{chain}
&\left(\frac{\pi}{2}\right)^{-\frac{d}{4}}\int\limits_0^T\delta(t,R_0)\sum\limits_{|n|\leq N}||c_n\phi_n(t,\cdot)||_{L^2(B_{R_0}^c)}\,dt\leq \\&\int\limits_0^T\left(\int\limits_{\omega}\left|\sum\limits_{|n|\leq N}c_n\phi_n(t,x)\right|^2\right)^{\frac{1}{2}}\,dt\leq \nonumber T^{\frac{1}{2}}\left(\int\limits_0^T\int\limits_{\omega}\left|\sum\limits_{|n|\leq N}c_n\phi_n(t,x)\right|^2\,dx\,dt\right)^{\frac{1}{2}}.
\end{align} 
\end{lem}
\begin{proof}
For the first inequality, we note that $c_n$ is real valued and $\gamma(t)$ is the same factor for all the propagated wavepackets then applying Lemma \ref{ass1} gives the first inequality, provided $\Omega\subset B_{R_0}$ and $v_n=c_n\phi_n(t,x)$ and $\tilde{\delta}|\gamma(t)|=\delta$. The last inequality in the Lemma follows by Cauchy Schwartz. 
\end{proof} 

The $\ell^1$ norm of the $c_n's$ can be related to the $L^2(\mathbb{R}^d)$ norm of $f$ as follows. 
\begin{lem}\label{chainlem2}
We have that for $f$ in $\mathcal{A}$ that there are $c_n$ real positive numbers and $\phi_n(t,x)$ as in Lemma \ref{gaussianwavepackets} so that 
\begin{align*}
\|f\|_{L^2(\mathbb{R}^d\,dx)} \leq \left(\frac{\pi}{2}\right)^{\frac{d}{4}}\sum\limits_{|n|\leq N}|c_n|+\eta.
\end{align*}
\end{lem}
\begin{proof}
Using Minkowski's inequality and the normalization of the $\phi_n(0)'s$ we see 
\begin{align*}
\left(\int\limits_{\mathbb{R}^d}|\sum\limits_{|n|\leq N}|c_n\phi_n(0,x)|^2\,dx\right)^{\frac{1}{2}}\leq \sum\limits_{|n|\leq N}\left(\int\limits_{\mathbb{R}^d}|c_n\phi_n(0,x)|^2\,dx\right)^{\frac{1}{2}}= \left(\frac{\pi}{2}\right)^{\frac{d}{4}}\sum\limits_{|n|\leq N}|c_n|.
\end{align*}
The last inequality comes from assuming that $f\in\mathcal{A}$. 
\end{proof} 
\begin{proof}[Proof of Theorem \ref{main}]
Assume that $f$ satisfies Assumption \ref{as1} so that for some positive real numbers $c_n$ we can write
\begin{align}\label{repg}
f(x)=\sum\limits_{|n|\leq N}c_n\phi_n(x)+\eta(x)
\end{align}
with $\eta(x)$ having $L^2(\mathbb{R}^d)$ norm less than $\eta$. (The equality is understood in the sense of $ L^2(\mathbb{R}^d)$ norms). Moreover from Lemma \ref{gaussianwavepackets} and linearity we have 
\begin{align}\label{reppg}
u(t,x)=\sum\limits_{|n|\leq N}c_n\phi_n(t,x)+\eta(t,x)
\end{align} 
with $\eta(t,x)$ having $L^2(\mathbb{R}^d)$ norm less than $\eta$ for all $0\leq t\leq T$. Now we proceed by a bit of a bootstrap argument. We additionally assume that $a_n=\varepsilon n$ are chosen such that Lemma \ref{lowerinner} holds by the hypothesis on $\varepsilon$.  By application of equalities \eqref{repg} and \eqref{reppg} followed by Lemmata \ref{lowerinner}, \ref{chainlemma}, and \ref{chainlem2} in direct succession we have that 
\begin{align*}
&T^{\frac{1}{2}}\left(\|u \|_{L^2((0,T)\times\omega)}+\eta T\right)\\&
\geq T^{\frac{1}{2}}\left(\int\limits_0^T\int\limits_{\omega}|\sum\limits_{|n|\leq N}c_n\phi_n(t,x)|^2\,dx\,dt\right)^{\frac{1}{2}} \\&
\geq \left(\frac{\pi}{2}\right)^{-\frac{d}{4}}\left(\int\limits_0^T\epsilon(t,R)\delta(t,R_0)\,dt\right)\left(\sum\limits_{|n|\leq N}|c_n|\right) \\&
\geq \left(\frac{\pi}{2}\right)^{-\frac{d}{2}}\left(\int\limits_0^T\epsilon(t,R)\delta(t,R_0)\,dt\right)(\|f\|_{L^2(\mathbb{R}^d)}-\eta). \\&
\end{align*}
This concludes the proof of the theorem. The constant in the inequality is defined as
\begin{align}\label{CT}
&C_T=\left(\frac{\pi}{2}\right)^{\frac{d}{2}}T^{\frac{1}{2}}\left(\int\limits_0^T\epsilon(t,R)\delta(t,R_0)\,dt\right)^{-1}\\& \nonumber
=T^{\frac{1}{2}}\left(\int\limits_0^T\frac{e^{\frac{1}{2}}A^{\frac{d-1}{4}}R_0^{\frac{d-1}{2}}}{\pi^{\frac{d+2}{4}}2^{\frac{d+3}{2}}}e^{-A(R_0^2+R^2)}\,dt\right)^{-1}.
\end{align}
\end{proof}

\begin{proof}[Proof of Theorem \ref{main3}]
Let $a_n$ be in $\mathbb{R}^d_+$ and $c_n\in \mathbb{R}_+$. By the hypothesis on the initial data $u_0(x)$ that it is in $\mathcal{A}$ we have that 
\begin{align}
u_0(x)=\sum\limits_{|n|\leq N}c_ne^{-|x+a_n|^2}=\sum\limits_{|n|\leq N}c_n\phi_n(0,x)+\eta(x)
\end{align}
in the sense of $L^2(\mathbb{R}^d)$ norms for some $\eta(x)$ with norm less than $\eta$. 
This allows us to write
\begin{align}\label{string}
&||u(t,\cdot)||^2_{L^2(\omega)}=\sum\limits_{n,m}c_nc_m\langle \phi_n(t,\cdot),\phi_m(t,\cdot)\rangle_{L^2(\omega)}+\eta'
\end{align}
where $\eta'\leq \eta^2$. 
Let $$|\alpha_N|=\max\limits_{n,m}|a_n-a_m|.$$ Let $a_{n'}$ be such that $$R_1=|a_{n'}|=\max\limits_m|a_m|.$$ We consider $a_{n'}$ as the new origin in the coordinate plane. Under this coordinate change we have that $\phi_n$ becomes $\tilde{\phi}_n$ where $\tilde{\phi}_n(t,x)$ have $\tilde{a}_n$ as their centers. The number $R_2$ we define as $R_2=R_1-\frac{\mathrm{diam}(\Omega)}{2}$ which is the distance from this new center $a_{n'}$ to the side of the ball enclosing $\Omega$. 

\begin{figure}[h!]
 \begin{center}
 \includegraphics*[width=10cm]{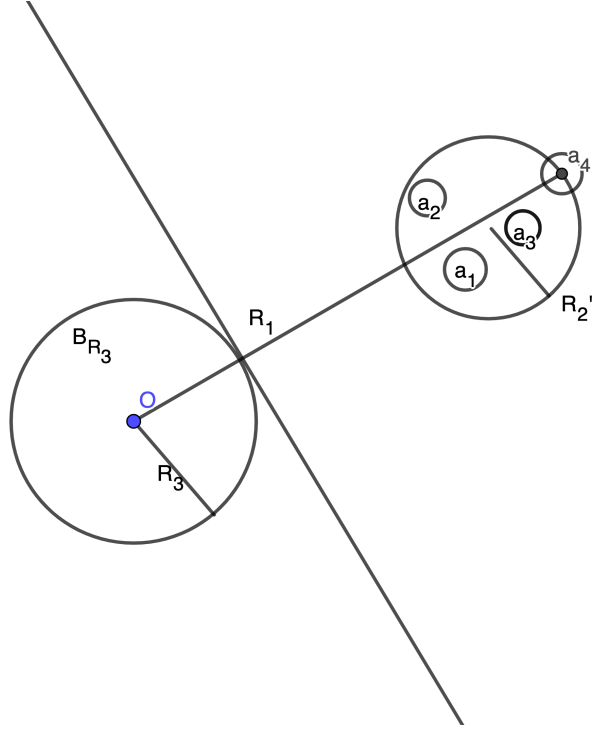}
 \end{center}
 \caption{The Gaussians at time $t=0$ are centered at $a_n$ $n=1,2,3,4$ and have their $L^2(\mathbb{R}^2)$ norm largely within a ball of radius $\frac{1}{\sqrt{2}}$. For $\Omega=B_{R_3}$, then we see if $R_1-R_3=R_2$ and $2R_2'\leq R_2$ these smaller balls do not intersect $B_{R_3}$, or the plane at a distance $R_2$ from $a_4$ containing it. In this case, $R_2'\geq \frac{1}{2}\max\limits_{n,m}|a_n-a_m|+\frac{1}{\sqrt{2}}$, and we should be able to compare the $L^2$ norms of the sum of the Gaussians over the area with the plane removed, to their $L^2(\mathbb{R}^d)$ norms. When the coordinates and radii of the balls are rescaled in time, the same principle still holds which is what is behind the proof of Theorem \ref{main3}.}
\end{figure}

Removing the plane containing the ball gives
\begin{align}\label{lower}
\sum\limits_{n,m}c_nc_m\langle \phi_n(t,\cdot),\phi_m(t,\cdot)\rangle_{L^2(\omega)}\geq \sum\limits_{n,m}c_nc_m\langle \tilde{\phi}_n(t,\cdot),\tilde{\phi}_m(t,\cdot)\rangle_{L^2(\mathbb{R}^d\setminus [R_2,\infty)^d)}.
\end{align}
We will show that
\begin{align}\label{hiB}
&2e|\int\limits_{[R_2,\infty)^d}\tilde{\phi}_n(t,x)\overline{\tilde{\phi}_m(t,x)}\,dx| 
\leq 2e\int\limits_{[R_2,\infty)^d}|\tilde{\phi}_n(t,x)\overline{\tilde{\phi}_m(t,x)}|\,dx\\& \leq \nonumber
\langle \tilde{\phi}_n(0,\cdot),\tilde{\phi}_m(0,\cdot)\rangle_{L^2(\mathbb{R}^d)}=\langle \phi_n(0,\cdot),\phi_m(0,\cdot)\rangle_{L^2(\mathbb{R}^d)}.
\end{align}
The main result will then follow since 
\begin{align}\label{hiB2}
&\sum\limits_{n,m}c_nc_m\langle \tilde{\phi}_n(t,\cdot),\tilde{\phi}_m(t,\cdot)\rangle_{L^2(\mathbb{R}^d\setminus [R_2,\infty)^d)}=\\& \nonumber\sum\limits_{n,m}c_nc_m\langle \tilde{\phi}_n(t,\cdot),\tilde{\phi}_m(t,\cdot)\rangle_{L^2(\mathbb{R}^d)}-\sum\limits_{n,m}c_nc_m\langle \tilde{\phi}_n(t,\cdot),\tilde{\phi}_m(t,\cdot)\rangle_{L^2([R_2,\infty)^d)}
\end{align}
combined with \eqref{lower} and \eqref{hiB} will give the result after integrating all the inequalities over $t$. In order to compute
\begin{align}\label{desired}
\int\limits_{[R_2,\infty)^d}|\tilde{\phi}_n(t,x)\overline{\tilde{\phi}_m(t,x)}|\,dx
\end{align}
we recall that
\begin{align*}
&\int\limits_{R_2}^{\infty} e^{-D_0|x-\tilde{c}|^2-D_0|x-\tilde{d}|^2}\,dx\\&=\left(\frac{\sqrt{\pi}}{2\sqrt{2D_0}}\right)^d\prod\limits_{i=1}^d\left(\mathrm{erfc}\left(\sqrt{2D_0}\left(R_2-\frac{\tilde{c_i}+\tilde{d_i}}{2}\right)\right)\right)e^{-\frac{D_0|\tilde{c}-\tilde{d}|^2}{2}}.
\end{align*}
Let $2D_0= A$ with $A$ as in \eqref{adef} and the vectors $\tilde{c}$ and $\tilde{d}$ we set equal to $\tilde{c}=-\tilde{a}_n/(y_2(t))$ and $\tilde{d}=-\tilde{a}_m/y_2(t)$ to match with \eqref{desired}. Using that $\mathrm{erfc}(x)\leq e^{-x^2}$ for all $x>0$ if we have that
\begin{align}
e^{-2D_0\left(R_2-\frac{\tilde{c_i}+\tilde{d_i}}{2}\right)^2}e^{-\frac{D_0|\tilde{c}_i-\tilde{d}_i|^2}{2}}\leq \sqrt{D_0}e^{-1}e^{-\frac{(a_{n_i}-a_{m_i})^2}{2}}
\end{align}
then \eqref{hiB} follows immediately. Whenever $D_0\geq 1$ we can find $R_2$ such that  
\begin{align}\label{if}
2D_0\left(R_2-\frac{\tilde{c_i}+\tilde{d_i}}{2}\right)^2>1-\log(\sqrt{D_0}).
\end{align}
However it is not guaranteed that $D_0\geq 1$. Therefore we take 
\begin{align*}
4D_0\left(R_2-\frac{\tilde{c_i}+\tilde{d_i}}{2}\right)^2>(a_n-a_m)^2+2-2\log(\sqrt{D_0})
\end{align*}
as our condition instead. We now have $-\tilde{c}=\frac{a_n-R_1}{y_2(t)}$ and $-\tilde{d}=\frac{a_m-R_1}{y_2(t)}$ by definition. This results in
 \begin{align*}
 R_2>\sqrt{\frac{|\alpha_N|^2+2-\log(\sqrt{D_0})}{4D_0}}+\frac{|\alpha_N|}{y_2(t)}
 \end{align*}
 We note that at $t=0, D_0=1, y_2(0)=1$ from \eqref{Riccati} and this would mean the expression for $R_2$ found from \eqref{if} would exactly coincide with the value found pictorially in Figure 1. 
 
 Plugging in $D_0\neq 0$ in terms of the Riccati equation solutions we see that this implies
 \begin{align}
 R_1>\sqrt{\frac{|\alpha_N|^2+2-2\log(\sqrt{|y_2(t)|(1+16(y_3(t))^2)^{-1}})}{4(y_2(t))^2(1+16(y_3(t))^2)^{-1}}}+\frac{|\alpha_N|}{y_2(t)}+\frac{\mathrm{diam}(\Omega)}{2}
 \end{align} 
 Taking the maximum of the right hand side in $t\in [0,T)$ with $T$ as in Definition \ref{flow}, this finishes the proof by combining \eqref{hiB} and \eqref{hiB2} and integrating over time. In this case $C_T=\frac{e}{(e-1)T}$ in Theorem \ref{main3}. 
\end{proof}

\begin{proof}[Demonstration of Counterexample \ref{main2}]
The proof consists of two counterexamples. Let $\delta>0$ be a constant. 
We consider as the initial data to \eqref{nonauto} the Gaussian 
\begin{align*}
e^{-|x+\delta|^2}
\end{align*}  
By direct computation in the notation of Lemma \ref{lowerinner} we obtain
\begin{align*}
&\int\limits_{[R,\infty)^d}|u(t,x)|^2\,dx=\left(\frac{\pi}{4A}\right)^{\frac{d}{2}}|\gamma(t)|^2 \left(\mathrm{erfc}\left(\sqrt{A}\left(\frac{\delta}{y_2}+R\right)\right)\right)^d
\end{align*} 
as $\delta\rightarrow \infty$, this goes to $0$ if $y_2>0$. As a result, no lower bound on $\|u\|_{L^2((0,T)\times\omega)}$ exists. Therefore it is necessary that $y_2$ change sign if $\omega$ is consisting of the right or left half space minus a compact set. The same set up can be use if $\omega$ is instead a bounded domain, which violates the condition of Definition \ref{geof}. For fixed $\delta$ and some finite $R_1,R_2$
\begin{align*}
&\int\limits_{\omega}|u(t,x)|^2\,dx\leq 
\int\limits_{[R_1,R_2]^d}|u(t,x)|^2\,dx \\&\leq 
\left(\frac{\pi}{4A}\right)^{\frac{d}{2}}|\gamma(t)|^2 
\left(\mathrm{erf}\left(\sqrt{A}\left(\frac{\delta}{y_2}+R_1\right)\right)-\mathrm{erf}\left(\sqrt{A}\left(\frac{\delta}{y_2}-R_2\right)\right)\right)^d
\end{align*}
which as $\delta$ goes to infinity is again $0$.
\end{proof} 

\section{Applications}\label{apps}

Many of the examples in this section are taken from the physics literature, see \cite{kim} for more background on these examples and others. The collection of these examples is the proof of Proposition \ref{maincorollary}. 
\begin{example}[The harmonic oscillator and free Schr\"odinger revisited]
For the free Schr\"odinger equation, $T_D=\infty$. Using \eqref{lowerbd}, the conditions for \eqref{req} and \eqref{R1} for $y_2(t)=1$ and $y_3(t)=-t$ are seen hold under the stronger conditions
\begin{align}
\varepsilon N((\sqrt{1+16t^2})e^{\frac{R_0^2+NR_0}{8t}}+\left(4N+4\right)e^{\frac{N^2}{8t}})\leq \frac{1}{10}e^{-\frac{4R_0^2}{1+16t^2}}
\end{align}
and
\begin{align}
R_2\geq \frac{\sqrt{(|\alpha_N|^2+2+2\mathrm{ln}\sqrt{1+16t^2})(1+16t^2)}}{2}+|\alpha_N|
\end{align}
respectively for all finite $t$. Notice that in \eqref{adef}
\begin{align}
A=\frac{2}{\sqrt{1+16t^2}} \quad A\leq 2
\end{align}
and therefore $\min\limits_{0\leq t\leq T}\epsilon(t,R)$ and $\min\limits_{0\leq t\leq T}\delta(t,R_0)$ as defined by \eqref{defepst} and \eqref{delta} are positive, so $C_T$ in \eqref{CT} is bounded. \\

For $\kappa_1(t)=\kappa_2(t)=\frac{1}{2}$, the standard harmonic oscillator, we have the hypothesis of Theorem 2 are satisfied for all time $0\leq t<\frac{\pi}{2}$. Notice that the construction and stability estimates could be extended past the cusp in this case using exercise 11.1 p.129 in \cite{grigissjostrand} to include all times $t$ in $(0,\pi)$ which is sufficient since it is known the operator has periodic $L^2((0,T)\times\omega)$ norm. 
Using \eqref{lowerbd}, notice that for this example the condition \eqref{req} would certainly be satisfied for $\varepsilon$ satisfying
\begin{align}
2\varepsilon N\left(e^{2R_0^2+2NR_0}+(2N+2)e^{2N^2}\right)\leq \frac{1}{10}e^{-4R_0^2}
\end{align}
uniformly for $t\in [0,\frac{\pi}{2})$. Also $R_1$ satisfies the inequality \eqref{R1} if
\begin{align}
R_2\geq 2\sqrt{|\alpha_N|^2+6}+|\alpha_N|
\end{align}
which is finite for finite $|\alpha_N|$.  Notice that in \eqref{adef}
\begin{align}
A=\frac{2}{\sqrt{1+15\sin^2(t)}} \quad\quad \frac{1}{2}\leq A\leq 2
\end{align}
and therefore $\min\limits_{0\leq t\leq T}\epsilon(t,R)$ and $\min\limits_{0\leq t\leq T}\delta(t,R_0)$ are positive, so $C_T$ in \eqref{CT} is bounded.
\end{example}

\begin{example}[Potentials]
Consider the Hamiltonian with $\sigma$ a constant
\begin{align*}
H(t)=\frac{|p|^2}{2(t+d)^a}+\frac{\sigma^2(t+d)^b|x|^2}{2}, \quad a,b,d>0
\end{align*}
The equations of motion give 
\begin{align*}
\frac{d^2x(t)}{dt^2}+\frac{a}{t+d}\frac{dx}{dt}+\sigma^2(t+d)^{b-a}x(t)=0. 
\end{align*}
The solutions are 
\begin{align*}
x(t)=C_1(t+d)^{\alpha}J_{\nu}(z)+C_2\quad (t+d)^{\alpha}Y_{\nu}(z) 
\end{align*}
with 
\begin{align*}
z=\beta (t+d)^{\gamma} \quad \alpha=(1-a)/2\quad \gamma=1+(b-a)/2 \quad \beta=\sigma/|\gamma| \quad \nu=|\alpha/\gamma|.
\end{align*}
where $C_1$, $C_2$ are chosen so $x(0)=1, x'(0)=0$. However if $\gamma=0$, which occurs when $a=b+2$ then in this case there are power law solutions 
\begin{align*}
x(t)=C_1(t+d)^{\alpha_-}+C_2(t+d)^{\alpha_+}
\end{align*} 
where 
\begin{align*}
\alpha_{\pm}=\frac{-(a-1)\pm\sqrt{(a-1)^2-4\sigma^2}}{2}
\end{align*}
unless $a=1\pm 2\sigma$. 
\end{example}

\begin{example}[Caldirola-Kanai oscillator \cite{cald,kanai}]\label{CK}
The Caldirola-Kanai oscillator which has an exponentially increasing mass and frequency, has been studied through various methods is one of the most typical time-dependent quantum systems whose exact quantum states are known. 
Let $\kappa_1(t)=\frac{e^{-2at}}{2}, \kappa_2(t)=\frac{e^{2at}\sigma^2}{2}$, $a>0$, corresponding to the Hamiltonian 
\begin{align*}
H=\frac{|p|^2}{2e^{2at}}+\frac{\sigma^2e^{2at}|x|^2}{2}
\end{align*} then the solutions to the Hamiltonian flow are given by 
\begin{align*}
&x(t)=e^{-at}\cos(\Lambda t)+\frac{a}{\Lambda}e^{-at}\sin(\Lambda t)\\&
p(t)=-\frac{a^2}{\Lambda}e^{at}\sin(\Lambda t)- \Lambda e^{at}\sin(\Lambda t)
\end{align*} 
where $\Lambda=\sqrt{\sigma^2-a^2}$, with $\sigma^2-a^2>0$.  If $\Lambda$ real, then $T_D=\frac{\pi}{2\Lambda}$, for example.  The conditions \eqref{req} and \eqref{R1} are then also nonempty using trigonometric identities.

When $\Lambda$ is imaginary, we can recalculate. Let $\lambda=\sqrt{a^2-\sigma^2}$ where $a^2-\sigma^2>0$. Then 
\begin{align}
&x(t)=e^{-at}\cosh(\lambda t)+\frac{a}{\lambda}e^{-at}\sinh(\lambda t)\\&
p(t)=-\frac{a^2}{\lambda}e^{at}\sinh(\lambda t)+ \Lambda e^{at}\sinh(\lambda t)
\end{align}
The equations for the FIO coefficients are then 
\begin{align*}
&y_1(t)=\frac{e^{2at}}{2}\frac{\left(\lambda-\frac{a^2}{\lambda}\right)\sinh(\lambda t)}{\cosh(\lambda t)+\frac{a}{\lambda}\sinh(\lambda t)}\\&
y_2(t)=\frac{\lambda e^{at}}{\lambda \cosh(\lambda t)+a\sinh(\lambda t)} \\&
y_3(t)=-\frac{\sinh (\lambda t)}{2a\sinh(\lambda t)+2\lambda \cosh(\lambda t)} \\&
\end{align*}
As you can see $a(t)=\sqrt{y_2(t)}$, and $T_D=\infty$. In this case we have that 
\begin{align}
&\frac{\lambda e^{2at}}{a(1+e^{at})}\leq |y_2(t)|\leq 2e^{2at}\\&
\frac{1}{4a(e^{\lambda t}+1)}\leq |y_3(t)|\leq \frac{e^{-\lambda t}}{2\lambda}
\end{align}
and
\begin{align}
A=\frac{\lambda e^{at}}{\sqrt{(a\sinh(\lambda t)+\lambda \cosh(\lambda t))^2+4\sinh^2(\lambda t)}}
\end{align}
so there exists constants $C_1, C_2$ depending on $a$ and $\lambda$ such that 
\begin{align}
C_1e^{-\lambda t}\leq A\leq C_2e^{at}
\end{align}
the conditions \eqref{req} and \eqref{R1} are true for a nonempty set of $\epsilon_0$ and $R_1$ for finite $T$.  
\end{example} 
\begin{proof}[Proof of Proposition \ref{maincorollary}]
The proof of Proposition \ref{maincorollary} follows directly from the examples above. 
\end{proof}

\section{Conclusions}\label{conclusions}
We have the following Corollary
\begin{cor}\label{linftyobs}
The function $\phi(x)$ constructed in Lemma \ref{stepg} shifted a distance from $\Omega$ specified in Theorem \ref{main3} with $2M=|\alpha_N|$ is $L^2\textendash L^{\infty}$ observable if 
\begin{align}
\eta=\left(2e^{-\frac{M^2}{4}}+M\Delta x\right)^d<\frac{e-1}{4e}{\|\phi\|_{L^2(\mathbb{R}^d)}}.
\end{align}
\end{cor}
\begin{proof}
We have from the proof of Theorem \ref{main3} that $\eta$ must satisfy the following 
 $$\eta=\left(2e^{-\frac{M^2}{4}}+M\Delta x\right)^d\min\{(C_TT),1\}<\frac{1}{4}{\|\phi\|_{L^2(\mathbb{R}^d)}}$$ $C_T=\frac{e}{(e-1)T}$ to be $L^2\textendash L^{\infty}$ observable, that is for the $\eta$ to be folded into the left hand side of \eqref{obs}. By the construction of $\phi$, $2M=|\alpha_N|$ which gives the distance $R_1$ required by Theorem \ref{main3}. 
\end{proof}
More general $\phi\in\mathcal{A}$ can also be constructed. Because the FIO construction presented in Lemma \ref{FIOlem} here also gives an exact construction for initial data which consists of Hermite functions, there is room for future development and examples of the construction presented here.

\end{document}